\documentclass[12pt]{amsart}

\usepackage{microtype}

\usepackage{mathpazo} 

\usepackage{hyperref}
\hypersetup{
    colorlinks=true,
    linkcolor=Mahogany,      
    urlcolor=cyan,
    citecolor=blue,
}

\usepackage[hyperpageref]{backref}  

\usepackage[abbrev,alphabetic]{amsrefs} 

\usepackage{amssymb} 
\usepackage[dvipsnames]{xcolor}
\usepackage{multirow}
\usepackage{pgf}
\usepackage{tikz-cd}
\usepackage[style=english]{csquotes}
\usepackage{url}

\newtheorem{thm}{Theorem}[section] 
\newtheorem*{thm*}{Theorem}
\newtheorem{cor}[thm]{Corollary}
\newtheorem{prop}[thm]{Proposition}
\newtheorem{conj}[thm]{Conjecture}
\newtheorem{lem}[thm]{Lemma}

\theoremstyle{definition} 
\newtheorem{defn}[thm]{Definition}
\newtheorem{ex}[thm]{Example} 
\newtheorem{con}[thm]{Construction} 
\newtheorem{ques}[thm]{Question}

\theoremstyle{remark}
\newtheorem{rem}[thm]{Remark}
\newtheorem{notation}[thm]{Notation}

\numberwithin{equation}{section}

\newcommand{\rk}[0]{\operatorname{rk}}
\newcommand{\codim}[0]{\operatorname{codim}}

\newcommand{\reg}{{\rm{reg}}}

\newcommand{\Aut}{\mathrm{Aut}}
\newcommand{\Simp}{\mathrm{Sim}}

\newcommand{\R}{\mathbb{R}}
\newcommand{\Z}{\mathbb{Z}}
\newcommand{\C}{\mathbb{C}}
\newcommand{\Q}{\mathbb{Q}}
\newcommand{\mP}{\mathbb{P}}
\newcommand{\B}{\mathbb{B}}

\makeatletter
\newcommand*\bigcdot{\mathpalette\bigcdot@{.5}}
\newcommand*\bigcdot@[2]{\mathbin{\vcenter{\hbox{\scalebox{#2}{$\m@th#1\bullet$}}}}}
\makeatother

\title[Inequalities of Miyaoka-Yau Type \& Uniformisation]{Inequalities of Miyaoka-Yau Type \& Uniformisation of varieties of intermediate Kodaira Dimension}

\author{Niklas M\"uller}
\address{Department of Mathematics, Universit\"at Freiburg,
Ernst-Zermelo-Str. 1, 79104 Freiburg, Germany.}
\email{{\tt niklas.mueller@math.uni-freiburg.de}}

\date{\today}



\begin{document}


    \begin{abstract}
        In this paper we present, for any integers $0\leq \nu \leq n$, a set of inequalities satisfied by the Chern classes of any minimal complex projective variety of dimension $n$ and numerical dimension $\nu$. 
        In the cases where $\nu$ is either very small or very large compared with $n$, this recovers many previously known results. 
        We demonstrate that our inequalities are sharp by providing an explicit characterisation of those varieties achieving the equality; our proof, in particular, resolves the Abundance conjecture in this situation. 
        Additionally, we provide some new examples of varieties with extremal Chern classes that demonstrate the optimality of our results.
    \end{abstract}
    
    \maketitle


    \setcounter{tocdepth}{2}
    \tableofcontents

    
    \section{Introduction}
\label{section-Introduction}

\renewcommand{\thethm}{\Alph{thm}}
\renewcommand{\theequation}{\arabic{equation}}

Recall the following well-known result which is an immediate consequence of the Enriques-Kodaira classification, see \cite[§ VI]{BarthHulekPetersVen_CompactSurfaces}, and work of 
Miyaoka \cite{Miyaoka_ChernNumbersSurfacesGeneralType, Miyaoka_SurfacesWith3c2=c1^2} and Yau \cite{Yau_KEMetrics}.

\begin{thm*}\emph{(Enriques, Kodaira, Miyaoka, Yau)}
\label{INTRO-thm-surfaces}

    \noindent
    Let $X$ be a minimal smooth projective surface which is not uniruled. Then
    \begin{equation}
        3c_2(X) \geq c_1(X)^2.
        \label{INTRO-eq-BMY}
    \end{equation}
    Moreover, the equality in \eqref{INTRO-eq-BMY} is attained if and only if there exists a finite \'etale cover $\pi\colon X '\rightarrow X$ by one of the following varieties:
   \begin{itemize}
       \item[\emph{(1)}] $X'= A$ is an Abelian variety.
       \item[\emph{(2)}] $X'\cong E\times C$ is the product of an elliptic curve $E$ and a curve $C$ of genus $g(C) \geq 2$.
       \item[\emph{(3)}] $X'\cong \B^2/\Lambda$ is a quotient of the complex unit ball $\B^2 := \{ z\in \C^2 \ | \ \| z \| <1\}$.
   \end{itemize}
\end{thm*}
Here, $c_i(X) := c_i(\mathcal{T}_X) \in H^{2i}(X, \C)$ denote the \emph{Chern classes} of $X$ \cite[§ 3]{Fulton_IntersectionTheory}.
Note that the cases $(1), (2)$ and $(3)$ above are mutually exclusive and can be distinguished, for example, according to the \emph{numerical dimension} $\nu(X)$ of $X$:
\[
\nu(X) := \max\{ k\ |\ c_1(X)^k \neq 0 \in H^*(X, \C) \}. 
\]
The conclusion of the Theorem is rather surprising as, when $X$ is a surface, $c_2(X)$ and $c_1(X)^2$ depend only on the $C^\infty$-diffeomorphism type of $X$ \cite{Kotschick_ChernNumbersAndDiffeos}.

The natural question arises, to what extent it is possible to obtain similar results for varieties of higher dimension. Note that the minimal model program indicates that the natural analogue of minimal surfaces in higher dimensions are log terminal (aka midly singular) projective varieties $X$ whose canonical divisor $K_X$ is nef. 

Indeed, such questions have been studied at various places throughout the literature over the past years, see Section \ref{ssec-Comparison-previous-works} for a detailed discussion. 
However, so far, most papers on the subject have addressed these questions only under additional, restrictive assumptions on the numerical dimension $\nu(X)$ of $X$. Classically, the cases $\nu(X)= 0$ and $\nu(X) = \dim X$ have been particularly closely investigated, cf.\ the discussion in Section \ref{ssec-Comparison-previous-works}. 
Although results have been recently put forth addressing the cases $\nu(X) = 1$ \cite{IwaiMatsumura_ManifoldsWithC2=0, IMM_3c2=c1^2}, $\nu(X) = 2$ \cite{IMM_3c2=c1^2}, and $\nu(X) = n-1$ \cite{HaoSchreieder_EqualityMYinequalityNonGeneralType}, see also \cite{PeternellWilson_ThreefoldsExtremalChernNumbers}, the range $3\leq \nu(X)\leq \dim X -2$ had been completely uninvestigated so far. Thus, in this paper, we focus on the case of \emph{intermediate} numerical dimension.

\subsection{Summary of main results}
\label{ssec-main-results}

The following is our main result:
\begin{thm}
\label{INTRO-thm-statement-light}
    Let $X$ be a log terminal projective variety of dimension $n$. Assume that $K_X$ is nef, of numerical dimension $\nu \geq 2$. Let $H$ be an ample Cartier divisor on $X$. Then there exists a number $\varepsilon_0 >0$ such that
    \begin{equation}
        P_X(\varepsilon) := \Big(2(\nu+1)c_2(X) - \nu c_1(X)^2\Big)\cdot (K_X + \varepsilon H)^{n-2} \geq 0, \qquad \forall\ 0 \leq \varepsilon < \varepsilon_0.
        \tag{$\star$}
        \label{INTRO-eq-Chern-class-inequality}
    \end{equation}
    Moreover, equality holds in \eqref{INTRO-eq-Chern-class-inequality} for some $0<\varepsilon < \varepsilon_0$ and some $H$ if and only if $X$ admits a finite quasi-\'etale Galois cover $\pi\colon X'\rightarrow X$ by a product $X'\cong A \times B$ of an Abelian variety $A$ and a smooth ball quotient variety $B \cong \B^\nu/\Lambda$. In particular, in this situation $X$ has only finite quotient singularities and $K_X$ is semiample.
\end{thm}
Here, we denote by $c_i(X) := \hat{c}_i(X)$ the \emph{orbifold Chern classes} of $X$ \cite{GKPT_MY_Inequality_Uniformisation_of_Canonical_models} and by $\B^n := \{z\in \C^n\ | \|z\|^2<1 \}$ the $n$-dimensional \emph{complex unit ball}. We emphasise that Theorem \ref{INTRO-thm-statement-light} relies on and extends results by Lu--Taji \cite{LuTaji_QuasiEtaleQuotientsAbelianVarieties}, Greb--Kebekus--Peternell--Taji \cite{GKPT_HarmonicMetricsUniformisation} and Guenancia--Taji \cite{GuenanciaTaji_SemistabilityLogCotangentSheaf}, see Section \ref{ssec-Comparison-previous-works} for a detailed discussion. Note that when $\nu \leq 2$, similar results have been obtained in \cite{LuTaji_QuasiEtaleQuotientsAbelianVarieties, IwaiMatsumura_ManifoldsWithC2=0, IMM_3c2=c1^2}.

We emphasise that the semiampleness of $K_X$ is not merely a consequence of the statement of Theorem \ref{INTRO-thm-statement-light},
but a central ingredient in its proof, see Section \ref{sec-uniformisation-v>=2-Semiampleness}.

Let us briefly comment on the nature of condition \eqref{INTRO-eq-Chern-class-inequality}. Indeed, we encourage the reader to view \eqref{INTRO-eq-Chern-class-inequality} as iterative. To explain what we mean, recall the following elementary fact: let $P = \sum_{i} a_i \cdot X^i \in \R[X]$ be a real polynomial. Then there exists a number $\varepsilon_0 >0$ such that
\begin{align}
    P(\varepsilon) \geq 0, \quad \forall \ 0 < \varepsilon \leq \varepsilon_0
    \label{eq-inequality-polynomials}
\end{align}
if and only if either $P = 0$ or there exists an integer $k\geq 0$ such that $a_0 = \ldots = a_{k-1} = 0$ and $a_k>0$. Moreover, in this case one can choose $\varepsilon_0$ in such a way that $P(\varepsilon) = 0$ for some $0 < \varepsilon \leq \varepsilon_0$ if and only if $P = 0$. In our case, the coefficients of $P_X$ are given by
\begin{align}
   a_{X, i} := \Big(2(\nu+1)c_2(X) - \nu c_1(X)^2\Big)\cdot K_X^{n-i-2}\cdot H^{i};
   \label{INTRO-eq-coefficients}
\end{align}
in particular, the equality in \eqref{INTRO-eq-Chern-class-inequality} is equivalent to the vanishing of all of these numbers. One might wonder what can be said about the structure of $X$ in case we only assume the first $n-k$ numbers in \eqref{INTRO-eq-coefficients} to vanish. This is the content of the following result:

\begin{thm}
\label{INTRO-thm-statement-strong}
    Let $X$ be a log terminal projective variety of dimension $n$ such that $K_X$ is semiample. Let $H$ be an ample Cartier divisor on $X$. Assume that there exists an integer $2\leq k\leq \nu := \nu(X)$ such that
    \[
    \Big(2(\nu+1)c_2(X)  - \nu c_1(X)^2\Big)\cdot K_X^{i-2}\cdot H^{n-i} = 0, \quad \forall i=k, \ldots, n.
    \]
    Then there exists an Abelian variety $A$ of $\dim A = n - \nu$, a smooth ball quotient variety $B = \B^\nu/\Lambda$ of $\dim B = \nu$, a finite subgroup $G\subseteq \Aut(A\times B)$ acting on $A\times B$ without fixed points in codimension one, and a birational map
    \[
    \varphi\colon X \dashrightarrow (A\times B)/G.
    \]
    In particular, the canonical model $X_{\mathrm{can}} \cong B/G$ is a ramified ball quotient. Furthermore, there exists a $G$-invariant open subset $U\subseteq B$, with complement of codimension at least $k+3$, such that $\varphi|_U\colon X|_U \rightarrow (A\times U)/G$ is an isomorphism.
\end{thm}

The following charcaterisation of complex ball quotients purely in terms of their Chern numbers is an immediate consequence; it extends \cite{Yau_KEMetrics,GKPT_HarmonicMetricsUniformisation}, where $K_X$ was assumed to be ample:
\begin{cor}\emph{(cf.\ \cite{Yau_KEMetrics,GKPT_HarmonicMetricsUniformisation})}
\label{INTRO-cor-ball-quotients}
    Let $X$ be a log terminal projective variety with nef canonical class $K_X$. Assume that $\nu(X) \geq 2$. Then $X$ is isomorphic to a ball quotient variety if and only if there exists an ample Cartier divisor $H$ on $X$ such that
    \[
    \Big(2(n+1)c_2(X) - nc_1(X)^2\Big)\cdot K_X^{i-2}\cdot H^{n-i} = 0, \quad \forall \ i = 2, \ldots, n.
    \]
\end{cor}
\noindent
Let us end this summary by emphasising that Corollary \ref{INTRO-cor-ball-quotients} indeed provide genuinely new insights compared to \cite{GKPT_HarmonicMetricsUniformisation} as the following example shows:
\begin{thm}
\label{INTRO-thm-counterexample-catanese}
    There exists a log terminal projective variety $X$ of dimension $4$ such that $K_X$ is big and nef but not ample and such that
    \[
    \Big(2\cdot (4 + 1)c_2(X) - 4 c_1(X)^2\Big) \cdot K_X^2 = 0.
    \]
\end{thm}
\noindent
Observe that $X$ cannot be isomorphic to a ball quotient variety, for otherwise $K_X$ would be ample. To the best of the author's knowledge, such an example has not previously appeared in the literature.

\subsection{Previous results}
\label{ssec-Comparison-previous-works}

Chern class inequalities and uniformisation results have a long history in complex algebraic geometry, see \cite{GKT_OverviewUniformisation} for a comprehensive survey. For $K$-trivial surfaces, the inequality \eqref{INTRO-eq-BMY} is classical, while for surfaces of numerical dimension one it is a consequence of Kodaira's work on elliptic fibrations. For surfaces with ample canonical bundle, \eqref{INTRO-eq-BMY} was proved independently by Miyaoka \cite{Miyaoka_ChernNumbersSurfacesGeneralType} and Yau \cite{Yau_KEMetrics}, following earlier works by Van de Ven \cite{vdVen_MYinequality} and Bogomolov \cite{Bogomolov_MYinequality}. For arbitrary log terminal surfaces of general type it was proved by Miyaoka \cite{Miyaoka_SurfacesWith3c2=c1^2}. Later, Langer \cite{Langer_MILogCanonicalSurfaces, Langer_MILogCanonicalSurfacesII} managed to also treat the log canonical case.

As was indicated above, in the study of Chern classes on higher dimensional minimal varieties $X$, the cases when $\nu(X) \in \{0, \dim X\}$ have received particularly much attention over the past years. When $\nu(X) = 0$, Miyaoka \cite{miyaoka_MiyaokaInequality} showed that $c_2(X)\cdot H^{n-2} \geq 0$ for any ample divisor $H$ on $X$, which is a natural analogue of \eqref{INTRO-eq-Chern-class-inequality} in this situation. Moreover, it is well-known that the equality $c_2(X) \cdot H^{n-2} = 0$ holds if and only if $X$ is a finite \'etale quotient of an Abelian variety, see e.g.\ \cite{Yau_KEMetrics}. For log terminal varieties, the analogues statement was proved by Lu--Taji \cite{LuTaji_QuasiEtaleQuotientsAbelianVarieties}, building on earlier work by 
Shepherd-Barron--Wilson \cite{SBW_SingularThreefoldsAbelianVarieties} and Greb--Kebekus--Peternell \cite{gkp_QuasiEtaleCovers}, see also \cite{CGG_UniformisationOrbifoldPairs} for a recent generalisation to the context of klt pairs. For smooth projective varieties $X$ with ample canonical class, Yau \cite{Yau_KEMetrics} proved the inequality $(2(n+1)c_2(X) - n c_1(X)^2)\cdot K_X^{n-2} \geq 0$ and he showed that in the equality case, $X$ must be a quotient of the complex unit ball. Subsequently, analogous results for log terminal varieties were obtained by Guenancia--Taji \cite{GuenanciaTaji_SemistabilityLogCotangentSheaf}, Greb--Kebekus--Peternell-Taji \cite{GKPT_MY_Inequality_Uniformisation_of_Canonical_models, GKPT_HarmonicMetricsUniformisation}
and Claudon--Guenancia--Graf \cite{CGG_UniformisationOrbifoldPairs}, following ideas of Simpson \cite{Simpson_VHSandUniformisation}. See also  \cite{CataneseScala_PolydiscQuotients,CataneseScala_PolydiscQuotientsII, Patel_UniformisationBuondedSymmetricDomains, GrafPatel_UniformisationByBoundedSymmetricDomains, Catanese_BallQuotientsAndSpecialTensors, ZZZ_MYKaehlerSpaces, IJZ_MYInequality_Big} for some recent related developments.

For minimal projective varieties $X$ of arbitrary numerical dimension, Miyaoka \cite{miyaoka_MiyaokaInequality} proved the famous inequality 
\begin{equation}
    (3c_2(X) - c_1(X)^2)\cdot H^{n-2}\geq 0,
    \label{INTRO-eq-Miyaoka}
\end{equation}
which corresponds to the case $\nu=2$ in \eqref{INTRO-eq-Chern-class-inequality} (note that by \eqref{INTRO-eq-Chern-class-inequality}, equality can never hold in \eqref{INTRO-eq-Miyaoka} when $\nu> 2$), see also \cite{Langer_NoteOnBogomolovInstability} for a shorter proof. Peternell--Wilson \cite{PeternellWilson_ThreefoldsExtremalChernNumbers} classified minimal projective threefolds achieving the equality in \eqref{INTRO-eq-Miyaoka}, thereby essentially proving Theorem \ref{INTRO-thm-statement-light} in case $n=3$. More recently, Hao and Schreieder \cite{HaoSchreieder_EqualityMYinequalityNonGeneralType} classified minimal projective varieties $X$ such that the Kodaira dimension $\kappa(X) = n-1$ and which satisfy the equation $c_2(X)\cdot c_1(X)^{n-2} =0$. This result is closely related to Theorem \ref{INTRO-thm-statement-strong} in the case $\nu(X) = n-1$. In a similar spirit, Iwai--Matsumura \cite{IwaiMatsumura_ManifoldsWithC2=0} classified smooth projective varieties such that $K_X$ is nef and $c_2(X) \equiv 0$. The latter result was generalised in \cite{IMM_3c2=c1^2}, where the author, in collaboration with Iwai--Matsumura, proved Theorem \ref{INTRO-thm-statement-light} in case $\nu(X) = 2$. Apart from the results cited above, there is an extensive body of further research that has been conducted on similar and related topics. We refer the interested reader to \cite{Mul_Thesis} and specifically the respective introductions of the individual chapters, where a much more thorough survey of the literature on the topic was conducted.

\subsection{Main techniques and overview}
In Section \ref{sec-preliminaries}, we review some basic facts about finite covers, Higgs sheaves and semistability and we summarise the work of \cite{miyaoka_MiyaokaInequality,Langer_NoteOnBogomolovInstability} on Chern class inequalities. In Section \ref{sec-twisted-semistability-cotangent-bundles} we make some progress on a question of Langer \cite[Quest.\ 5.4]{Langer_NoteOnBogomolovInstability}; this is then used to prove the Inequality \eqref{INTRO-eq-Chern-class-inequality} in Section \ref{sec-Higher-Inequalities}. In Section \ref{sec-uniformisation-v>=2-Semiampleness}, we show that under the assumptions of Theorem \ref{INTRO-thm-statement-light}, $K_X$ is semiample. The proofs of Theorem \ref{INTRO-thm-statement-strong} and Theorem \ref{INTRO-thm-statement-light} are then completed in Section \ref{sec-proof-of-thm-B}. Finally, Theorem \ref{INTRO-thm-counterexample-catanese} is proved in Section \ref{sec-examples}.

The main idea to prove \eqref{INTRO-eq-Chern-class-inequality} is to reduce it, using Miyaoka's strategy \cite{miyaoka_MiyaokaInequality}, to a suitable semistability-type condition for the cotangent bundle of minimal varieties, Conjecture \ref{conj:twisted-semistability-cotangent-sheaf}. This part of the proof is similar to the ideas outlined by Langer \cite{Langer_NoteOnBogomolovInstability}. While we are unable to verify Conjecture \ref{conj:twisted-semistability-cotangent-sheaf} in general, building on an idea of Enoki \cite{Enoki_TwistedSemistability}, we manage to prove it in certain cases of particular interest, crucially relying on results of Guenancia--Taji \cite{GuenanciaTaji_SemistabilityLogCotangentSheaf} to do so.

In the case equality holds in \eqref{INTRO-eq-Chern-class-inequality}, an analysis of the Harder--Narasimhan filtration of $\Omega_X$ together with the Non-Abelian Hodge correspondence for klt spaces \cite{gkpt_HodgeTheoremKltSpaces,GKPT_HarmonicMetricsUniformisation} imposes strong structural restrictions on $X$. This strategy is similar to the one used in \cite{IMM_3c2=c1^2} to treat the case $\nu=2$ and is visible in many publications on the subject, see for example \cite{miyaoka_MiyaokaInequality, Simpson_VHSandUniformisation, LuTaji_QuasiEtaleQuotientsAbelianVarieties,GKPT_HarmonicMetricsUniformisation, CGG_UniformisationOrbifoldPairs, LiuLiu_KawamtaMiyaokaInequalityFanos, LiuLiu_KawamtaMiyaokaInequalityFanosII, Dailly_FanosMYequality}. There are, however, two key additional difficulties compared with \cite{IMM_3c2=c1^2}: First, to show that, in the situation of Theorem \ref{INTRO-thm-statement-light}, $K_X$ is semiample. Second, to show that the Iitaka fibration is isotrivial. Note that when $\nu=2$ both statements readily follow from results of Pereira--Touzet \cite{PereiraTouzet_FoliationsVanishingChernClasses}, so this was not an issue in
\cite{IMM_3c2=c1^2}. In the general case, to adress the first issue, we identify the Iitaka fibration with a certain Shafarevich morphism, which exists by work of Eyssidieux \cite{Eyssidieux_LinearShafarevichConjecture}, see Section \ref{sec-uniformisation-v>=2-Semiampleness}. To do so, a key ingredient is Kawamata's partial resolution of the Iitaka conjecture \cite{Kawamata_MinimalModelsKodairaDimension}. The second issue is overcome in Proposition \ref{prop-uniformisation-over-complete-intersection-surface}, making key use of Ambro's work on the canonical bundle formula \cite{Ambro_CanonicalBundleFormula} and structural results for manifolds with nef cotangent bundle by H\"oring \cite{Hoering_ManifoldsNefCotangentBundle}.

The counterexample in Theorem \ref{INTRO-thm-counterexample-catanese} is a special case of a classical construction previously employed by Kazhdan \cite{Khazdan_SomeApplicationsWeilRepresentations} and Rapoport--Zink \cite{RZ_LokaleZetaFunktionen}.

\subsection{Notational conventions}
\label{ssec-Notation}

Throughout this text we work over the field $\C$ of complex numbers. 
We follow the standard notions and conventions in \cite{hartshorne_AlgebraicGeometry, kollarMori_BirationalGeometry, lazarsfeld_PositivityI, lazarsfeld_PositivityII}. The \emph{dual} of a coherent sheaf $\mathcal{F}$ on a variety $X$ will be denoted by $\mathcal{F}^\vee$.
To lighten the notation, we usually simply speak of (torsion-free, reflexive) sheaves on $X$, when we really mean (torsion-free, reflexive) coherent sheaves of $\mathcal{O}_X$-modules. Given two coherent sheaves $\mathcal{F}_1, \mathcal{F}_2$ on $X$ we denote $\mathcal{F}_1 [\otimes] \mathcal{F}_2 := (\mathcal{F}_1 \otimes \mathcal{F}_2)^{\vee\vee}$. In a similar spirit, we denote $\Omega_X^{[1]} := (\Omega_X^1)^{\vee\vee}$. Given a holomorphic map $f\colon Y \rightarrow X$, we set $f^{[*]}\mathcal{F}_1 := (f^*\mathcal{F}_1)^{\vee\vee}$.

\subsection*{Acknowledgments}

First and foremost, I would like to thank Daniel Greb for suggesting the questions treated in this article as a thesis problem and for supporting me, through many conversations and much encouragement, during the creation of this article. My sincere gratitude goes to Andreas Höring and Junyan Cao for enabling me to visit them in Nice, where part of the research presented in this article was conducted. I am grateful to both Universität Duisburg-Essen and Universität Freiburg for providing a stimulative working environment. Moreover, I would like to thank Stefan Kebekus, Vladimir Lazi\'c and Thomas Peternell for their sincere advise. Finally, I am grateful towards all my friends in Essen and elsewhere, mathematical or not, for their encouragement and their support.

During the creation of this article, I was partially supported by the DFG research training group 2553 ‘Symmetries and classifying spaces:
analytic, arithmetic and derived’.

    \renewcommand{\thethm}{\thesection.\arabic{thm}}
    \renewcommand{\theequation}{\thesection.\arabic{equation}}

    \section{Preliminaries}
    \label{sec-preliminaries}

    \subsection{Finite covers}

    \subsubsection{Quasi-\'etale covers}

    Let $X$ be a normal analytic variety.
    A \emph{finite cover} of $X$ is a finite and surjective morphism $\pi\colon Y\rightarrow X$ from another normal analytic variety $Y$.
    We say that $\pi$ is \emph{Galois} if there exists a finite group $G$ acting holomorphically on $Y$ such that $\pi$ is $G$-invariant and identifies $X\cong Y/G$ with the quotient of $Y$ by $G$. 
   
    As per usual, we denote by $\pi_1(X)$ the \emph{fundamental group} of $X$. If $\iota\colon U\hookrightarrow X$ is the inclusion of a Zariski open subset, then the natural push forward map
    \begin{equation}
        \iota_*\colon \pi_1(U) \twoheadrightarrow \pi_1(X)
    \label{eq-inclusion-pushforward-fundamental-groups}
    \end{equation}
    is surjective, see \cite[Prop.\ 2.10]{Kollar_ShafarevichMapsAutomorphicForms}. It is an isomorphism if, additionally, $X$ is smooth and $\codim (X\setminus U) \geq 2$. 
    \begin{defn}
        A morphism $\pi\colon X' \rightarrow X$ between normal analytic varieties is called \emph{quasi-\'etale} if it is quasi-finite and \'etale in codimension one, i.e.\ there exists a closed analytic subset $Z\subseteq X$ with $\codim Z\geq 2$ such that $\pi\colon X'\setminus \pi^{-1}(Z)\rightarrow X\setminus Z$ is \'etale.
    \end{defn}
    Clearly, the composition of two quasi-\'etale morphisms is quasi-\'etale. Conversely, if the composition of two finite morphisms is quasi-\'etale, then both of the morphisms are quasi-\'etale. Moreover, the Galois closure of a quasi-\'etale morphism is quasi-\'etale \cite[Thm.\ 3.7]{gkp_QuasiEtaleCovers}. It is well-known that finite quasi-\'etale covers of $X$ are in natural bijection with finite index subgroups of $\pi_1(X_{\reg})$, see \cite{gkp_QuasiEtaleCovers}.
    \begin{defn}
        Let $X$ be a normal analytic variety. We say that $X$ is \emph{maximally quasi-\'etale} if any quasi-\'etale cover of $X$ is in fact \'etale.
    \end{defn}
    By a fundamental result of Greb--Kebekus--Peternell, any log terminal quasi-projective variety $X$ admits some maximally quasi-\'etale Galois cover $\pi\colon X'\rightarrow X$.
    Note that in this case, $X'$ is automatically log terminal \cite[Cor.\ 2.43]{Kollar_SIngularitiesOfTheMMP}.

    \subsubsection{Finite Covers of Products of two Varieties}

    In this subsection, we prove the following technical but completely elementary result:
    \begin{prop}
    \label{prop-Galois-covers-Abelian-varieties-ball-quotients}
        Let $f\colon X\rightarrow Y$ be a proper fibration between normal  varieties. Assume that there exists a finite quasi-\'etale (and not necessarily Galois) cover $\pi\colon X'\rightarrow X$ by a smooth variety $X'$. Let
        \[\begin{tikzcd}
            X' \arrow{r}{\pi} \arrow[swap]{d}{g} & X \arrow{d}{f} \\
            Y' \arrow{r}{\psi} & Y
        \end{tikzcd}
        \]
        be the Stein factorisation and assume that there exists a proper variety $F'$ such that $X'\cong F'\times Y'$ over $Y'$.
        Then there exist finite \'etale covers $F''\rightarrow F$ and $Y''\rightarrow Y'$ such that the induced morphism
        \[
        \pi'\colon F''\times Y''\rightarrow X
        \]
        is finite, quasi-\'etale and Galois.
    \end{prop}
    Proposition \ref{prop-Galois-covers-Abelian-varieties-ball-quotients} will be useful in the proof of Theorem \ref{INTRO-thm-statement-strong}. Its content is that once we can find \emph{some} cover of a given variety which splits as a product then we can also find a \emph{Galois} cover which splits as a product. We caution the reader that in general, a finite \'etale Galois cover of a product of two varieties need not be a product itself, see \cite[Ex.\ 3.1.17]{Mul_Thesis} for an explicit example. In the proof of Proposition \ref{prop-Galois-covers-Abelian-varieties-ball-quotients} we will require the following easy result:
\begin{prop}
\label{prop-pushforward-fundamental-groups}
    In the situation of Proposition \ref{prop-Galois-covers-Abelian-varieties-ball-quotients}, there exists a natural group homomorphism $f_*\colon \pi_1(X_{\mathrm{reg}}) \rightarrow \pi_1(Y_{\mathrm{reg}})$ such that the diagram
     \[
     \begin{tikzcd}
        \pi_1\big(X'\big) \arrow[hookrightarrow]{r}{\pi_*} \arrow[swap]{d}{g_*} & \pi_1\big(X_{\mathrm{reg}}\big) \arrow{d}{f_*} \\
        \pi_1\big(Y'\big) \arrow[hookrightarrow]{r}{\psi_*} & \pi_1\big(Y_{\mathrm{reg}}\big) 
        \end{tikzcd}
    \]
    commutes.
\end{prop}
\begin{proof}
   Since $g$ is equidimensional, also $f$ is equidimensional. Set $U := Y_{\mathrm{reg}}$. Then $\codim(X\setminus f^{-1}(U)) \geq 2$, and it follows that $\pi_1(X_{\mathrm{reg}}) \cong \pi_1(f^{-1}(U))$, see \eqref{eq-inclusion-pushforward-fundamental-groups}. Consequently, we can simply take $f_*\colon \pi_1(X_{\mathrm{reg}})  \cong \pi_1(f^{-1}(U)) \rightarrow \pi_1(U)$.
\end{proof}

\begin{proof}[Proof of Proposition \ref{prop-Galois-covers-Abelian-varieties-ball-quotients}]
    The proof will proceed by constructing a tower of coverings of $X'$.
    Let $\pi^{(1)}\colon X^{(2)}\rightarrow X'$ be the Galois hull of $\pi\colon X'\rightarrow X$; then $\pi^{(1)}$ is quasi-\'etale, and, hence, \'etale by purity of the branch locus, see \cite[Thm.\ 12.3.3]{Nemeti_NormalSurfaceSingularities}. In particular, $X^{(2)}$ is a complex manifold. Let
     \[
     \begin{tikzcd}
        X^{(2)} \arrow{r}{\pi^{(1)}} \arrow[swap]{d}{h} & X' \arrow{d}{g} \\
        Y^{(2)} \arrow{r}{\psi^{(1)}} & Y'
        \end{tikzcd}
    \]
    be the Stein factorisation of the submersion $g\circ \pi^{(1)}$. Then $Y^{(2)}$ is a complex manifold, $\psi^{(1)}$ is finite \'etale and $h$ is a submersion. Unfortunatly, $X^{(2)}$ has no reason to be a product, so we are not yet done with the proof.

    Let $\psi^{(2)}\colon Y^{(3)}\rightarrow Y'$ and $\xi^{(2)}\colon F^{(3)}\rightarrow F'$ be the finite \'etale covers determined by the formulae
    \begin{align}
        \begin{split}
            \pi_1\Big(Y^{(3)}\Big) & = \pi_1\Big(X^{(2)}\Big) \cap\Big( \{1\} \times \pi_1\big(Y'\big)\Big) \quad \mathrm{and} \\
    \pi_1\Big(F^{(3)}\Big) & = \pi_1\Big(X^{(2)}\Big) \cap\Big( \pi_1\big(F'\big) \times \{1\}\Big).
        \end{split}
        \label{eq-Covers-products-3}
    \end{align}
    Here, the asserted finiteness of $Y^{(3)}\rightarrow Y'$ is due to the fact that the natural group homomorphism
    $\pi_1(Y')/\pi_1(Y^{(3)}) \hookrightarrow \pi_1(X')/\pi_1(X^{(2)})$ is injective. The analogous argument shows the finiteness of $F^{(3)}\rightarrow F$. 

    Now, choose a finite index normal subgroup $\Lambda \subseteq \pi_1(Y_{\mathrm{reg}})$ such that $\Lambda \subseteq \pi_1(Y^{(3)})$. Let $\psi^{(3)}\colon Y^{(4)} \rightarrow Y^{(3)}$ be the corresponding finite \'etale cover, so that
    \begin{equation}
        \pi_1\Big(Y^{(4)}\Big) = \Lambda.
        \label{eq-cover-products-lambda}
    \end{equation}
    Set $F^{(4)} := F^{(3)}$ and $X^{(4)} := F^{(4)}\times Y^{(4)}$. We claim that the subgroup
    \begin{equation}
        \pi_1\Big(X^{(4)}\Big)
        = \pi_1\Big(F^{(4)}\Big)\times \pi_1\Big(Y^{(4)}\Big) \subseteq 
        \pi_1(X_{\mathrm{reg}})
        \label{eq-Galois-covers-Abelian-varieties-ball-quotients-claim}
    \end{equation}
    is normal of finite index. Observe that if we can prove this claim, then the proposition immediately follows. Indeed, the inclusion \eqref{eq-Galois-covers-Abelian-varieties-ball-quotients-claim} corresponds to a finite quasi-\'etale Galois cover $F^{(4)} \times Y^{(4)} = X^{(4)} \rightarrow X$, as required. 

    Now, the assertion that $ \pi_1(X^{(4)})$ has finite index in $\pi_1(X_{\mathrm{reg}})$ is clear since $\pi_1(F^{(4)}) \subseteq \pi_1(F')$ and $\pi_1(Y^{(4)}) \subseteq \pi_1(Y')$ are respectively subgroups of finite index. It remains to show that \eqref{eq-Galois-covers-Abelian-varieties-ball-quotients-claim} is a normal subgroup. To this end, we claim that
    \begin{equation}
        \pi_1\Big(F^{(4)}\Big)\times \pi_1\Big(Y^{(4)}\Big)
        =: \pi_1\Big(X^{(4)}\Big) 
        = \big(f_*\big)^{-1}\Big(\pi_1\big(Y^{(4)}\big)\Big)
        \cap
        \pi_1\Big(X^{(2)}\Big).
        \label{eq-Galois-covers-Abelian-varieties-ball-quotients-=}
    \end{equation}
    Here, $f_*\colon \pi_1(X_{\mathrm{reg}}) \rightarrow \pi_1(Y_{\mathrm{reg}})$ was defined in Proposition \ref{prop-pushforward-fundamental-groups}.
    Since both groups on the right hand side of \eqref{eq-Galois-covers-Abelian-varieties-ball-quotients-=} are normal in $\pi_1(X_{\mathrm{reg}})$, see \eqref{eq-cover-products-lambda}, the assertion then immediately follows.

    Note that the inclusion $\subseteq$ is clear. Indeed, both $\pi_1(F^{(4)}) \times \{1\}$ and $\{1\}\times\pi_1(Y^{(4)})$ are subgroups of $\pi_1(X^{(2)})$ and, hence, so is the subgroup generated by them. Regarding the reverse inclusion, pick an element
    \[
    (\gamma, \delta)\in \pi_1\big(X^{(2)}\big)\subseteq \pi_1\big(F'\big)\times \pi_1\big(Y'\big)
    \]
    with $f_*(\gamma, \delta) = \delta \in \pi_1(Y^{(4)})$. Then, by \eqref{eq-Covers-products-3}, $(1, \delta) \in \pi_1(X^{(2)})$ and, consequently, $(\gamma, \delta)\cdot (1, \delta)^{-1} = (\gamma, 1) \in \pi_1(X^{(2)})$. We deduce that 
    \[
    \gamma\in \pi_1\Big(F^{(4)}\Big) := \pi_1\Big(F^{(3)}\Big) := \pi_1\Big(X^{(2)}\Big) \cap\Big( \pi_1\big(F'\big) \times \{1\}\Big),
    \]
    see \eqref{eq-Covers-products-3}. This finishes the proof that \eqref{eq-Galois-covers-Abelian-varieties-ball-quotients-=} is indeed an equality. As explained before, this shows that \eqref{eq-Galois-covers-Abelian-varieties-ball-quotients-claim} is a finite index normal subgroup, thereby concluding the proof of the proposition.
\end{proof}

\subsection{Higgs sheaves and semistability}
\label{ssec-Higgs-Sheaves-Semistability}

\subsubsection{\texorpdfstring{$\Q$}{Q}-Chern classes}
    Let $X$ be a normal projective variety of dimension $n$ with quotient singularities in codimension two (for example, $X$ could be log terminal \cite[Prop.\ 4.18]{kollarMori_BirationalGeometry}). Given reflexive coherent sheaves $\mathcal{E}, \mathcal{F}$ on $X$, Langer \cite{Langer_intersectiontheorychernclasses}, see also \cite{Mumford_OrbifoldChernClasses, GKPT_MY_Inequality_Uniformisation_of_Canonical_models}, defines symmetric $\Q$-multilinear forms
    \begin{align*}
        c_1(\mathcal{E})& \colon  \mathrm{N}^{1}(X)_{\Q}^{n-1} \rightarrow \Q, 
        \quad (L_1, \ldots, L_{n-1}) \mapsto c_1(\mathcal{E})\cdot L_1\cdots L_{n-1}, \\
        c_1(\mathcal{E})c_1(\mathcal{F})&\colon  \mathrm{N}^{1}(X)_{\Q}^{n-2} \rightarrow \Q, 
        \quad(L_1, \ldots, L_{n-2}) \mapsto c_1(\mathcal{E})c_1(\mathcal{F})\cdot L_1\cdots L_{n-2}, \\
        c_2(\mathcal{E})& \colon  \mathrm{N}^{1}(X)_{\Q}^{n-2} \rightarrow \Q, 
        \quad (L_1, \ldots, L_{n-2}) \mapsto c_2(\mathcal{E})\cdot L_1\cdots L_{n-2},
\end{align*}
called \emph{$\Q$-Chern classes}, and which satisfy the following properties:
\begin{itemize}
\item[(1)] If $\pi\colon Y \rightarrow X$ is a finite cover from a normal projective variety with quotient singularities in codimension two, then
\[
c_1(\pi^{[*]}\mathcal{E})\pi^*\cdot L_1\cdots \pi^*L_{n-1} 
= (\deg\pi)\cdot (c_1(\mathcal{E})\cdot L_1\cdots L_{n-1})
\]
for all $(L_1, \ldots, L_{n-1}) \in \mathrm{N}^{1}(X)_{\Q}^{n-1}$. Ditto for $c_1(\mathcal{E})c_1(\mathcal{F}),  c_2(\mathcal{E})$.
\item[(2)]  If $n>2$, then for any general member $D \in V$ in a basepoint free linear system $V \subseteq |\mathcal{L}|$ of some line bundle $\mathcal{L}$, it holds that
$$
c_1(\mathcal{E}) \cdot c_{1}(L)\cdot  L_2\cdots L_{n-1}
=
c_1(\mathcal{E}|_D)\cdot (L_2|_D)\cdots (L_{n-1}|_D),
$$
for all $(L_1, \ldots, L_{n-2}) \in \mathrm{N}^{1}(X)_{\Q}^{n-2}$.
Ditto for $c_1(\mathcal{E})c_1(\mathcal{F}),  c_2(\mathcal{E})$.
\item[(3)] If $\rk\mathcal{E}=1$, then $c_2(\mathcal{E})\equiv 0$. 
\end{itemize}
Here, and throughout this text, given a $\Q$-multilinear form $\alpha\colon \mathrm{N}^{1}(X)_{\Q}^{k} \rightarrow \Q$, we denote $\alpha\equiv 0$ to mean that $\alpha\cdot L_1\cdots L_k = 0$ for all tuples $(L_1, \ldots, L_k)\in \mathrm{N}^{1}(X)_{\Q}^{k}$.

In case $X$ is log terminal, the above notion coincides with the one in \cite{GKPT_MY_Inequality_Uniformisation_of_Canonical_models}, see \cite[Introduction]{Langer_intersectiontheorychernclasses}. In particular, both notions coincide if $\dim X = 2$. Since $\mathrm{N}^{1}(X)_{\Q}$ is generated by classes of (semi-)ample line bundles, from (2) and \cite[Lem.\ 3.18]{GKPT_MY_Inequality_Uniformisation_of_Canonical_models}, it follows that
$c_1(-)$ and $c_2(-)$ satisfies the usual relations for duals, direct sums and (reflexivised) tensor products of two reflexive sheaves. The reader be warned, however, that while $c_1(-)$ is additive on sequences of reflexive sheaves which are exact in codimension one, $c_2(-)$ does not quite satisfy the usual additivity relation on short exact sequences, cf.\ \cite[Lem. 2.13]{IMM_3c2=c1^2}. As per usual, we denote
    \[
    \Delta(\mathcal{E}) := 2\rk(\mathcal{E})\cdot  c_2(\mathcal{E}) - (\rk(\mathcal{E})-1) c_1(\mathcal{E})^2.
    \]
    
    Given $\Q$-Weil divisors $A, B$ on $X$ and classes $L_1, \ldots, L_{n-2}\in N_1(X)_\Q$ we denote
    \[
    A\cdot B \cdot L_1\cdots L_{n-2} := c_1(\mathcal{O}_X(A))\cdot c_1(\mathcal{O}_X(B)) \cdot L_1\cdots L_{n-2}.
    \]

\subsubsection{(Log-) Higgs sheaves}
    Throughout this subsection, let $X$ be a normal variety and let $D$ be a reduced integral Weil divisor on $X$.
    \begin{defn}
        A \emph{log-Higgs sheaf} on $(X, D)$ is a pair $(\mathcal{E}, \theta)$ consisting in a coherent sheaf $\mathcal{E}$ on $X$ and an $\mathcal{O}_X$-linear map
        \[
        \theta \colon \mathcal{E}\otimes \mathcal{T}_X(-\log D) \rightarrow \mathcal{E},
        \]
        such that the induced morphism 
        \[
        \bigwedge^2\mathcal{T}_{X}(-\log D) \overset{[\cdot, \cdot]}{\longrightarrow} \mathcal{T}_{X}(-\log D) \overset{\theta}{\rightarrow} \mathcal{E}\mathrm{nd}(\mathcal{E})
        \] 
        vanishes. Here, $[\cdot, \cdot]$ denotes the Lie-bracket.
    \end{defn}
    In case $D=0$, we simply speak of a \emph{Higgs sheaf}. A subsheaf $\mathcal{F}\subseteq\mathcal{E}$ is a \emph{Higgs subsheaf} if $\theta(\mathcal{F}\otimes \mathcal{T}_X(-\log D))\subseteq \mathcal{F}$. An $\mathcal{O}_X$-linear map $\gamma\colon (\mathcal{E}_1, \theta_1)\rightarrow (\mathcal{E}_2, \theta_2)$ is a \emph{morhism of Higgs sheaves} if $\theta_2 \circ (\gamma \otimes \mathrm{id}_{\mathcal{T}_X(-\log D))}) = \gamma\circ \theta_1$.
     \begin{ex}
        Let $(\mathcal{E}_1, \theta_1), (\mathcal{E}_2, \theta_2)$ be Higgs sheaves on $X$. Then also $(\mathcal{E}_1\oplus \mathcal{E}_2, \theta_1\oplus \theta_2)$ and $(\mathcal{E}_1\otimes \mathcal{E}_2, \theta_1\otimes \textmd{Id}_{\mathcal{E}_2} + \textmd{Id}_{\mathcal{E}_1} \otimes \theta_2)$ are Higgs sheaves. 
        Moreover, the formula
        $$
        \mathcal{H}\mathrm{om}(\mathcal{E}_1, \mathcal{E}_2)
        \otimes \mathcal{T}_{X}
        \rightarrow 
        \mathcal{H}\mathrm{om}(\mathcal{E}_1, \mathcal{E}_2), 
        \quad 
        \varphi\otimes v \mapsto \theta_2(\varphi(-), v) - \varphi(\theta_1(-, v))
        $$
        defines a natural Higgs sheaf structure on $\mathcal{H}\mathrm{om}(\mathcal{E}_1, \mathcal{E}_2)$. In particular, if $(\mathcal{E}, \theta)$ is a Higgs sheaf, then we have a natural Higgs sheaf structure on $\mathcal{E}^{\vee}$. Similar statements hold for log-Higgs sheaves.
    \end{ex}
    \begin{rem}
    \label{rem-comparison-defs-Higgs-sheaves}
        Our definition of Higgs sheaf is precisely the one given in \cite{Langer_BogomolovsInequalityHiggsSheavesNormalVarieties}; see \cite[Section 4]{Langer_BogomolovsInequalityHiggsSheavesNormalVarieties} for a discussion of the advantages of this convention.
    \end{rem}
    Now, let $L_1, \ldots, L_{n-1}$ be nef divisor on $X$ such that $\alpha:= L_1\cdots L_{n-1} \neq 0$. Recall, that the \emph{slope} of a coherent sheaf $\mathcal{F}$ on $X$ is defined to be
    \[
    \mu_{\alpha}(\mathcal{F}) := \frac{c_1(\mathcal{F})\cdot L_1\cdots L_{n-1}}{\rk\mathcal{F}}.
    \]
    \begin{defn}
        Let $L_1, \ldots, L_{n-1}$ be nef divisors on $X$ such that $\alpha:= L_1\cdots L_{n-1} \neq 0$. A (log-) Higgs sheaf $(\mathcal{E}, \theta)$ is 
        \begin{itemize}
            \item[(1)] $\alpha$-semistable if $\mu_\alpha(\mathcal{F}) \leq \mu_{\alpha}(\mathcal{E})$ for all Higgs subsheaves $\mathcal{F}\subseteq (\mathcal{E}, \theta)$.
            \item[(2)] $\alpha$-stable if $\mu_\alpha(\mathcal{F}) < \mu_{\alpha}(\mathcal{E})$ for all Higgs subsheaves $\mathcal{F}\subseteq (\mathcal{E}, \theta)$ with $0<\rk\mathcal{F}< \rk\mathcal{E}$.
            \item[(3)] $\alpha$-polystable if it is a direct sum of $\alpha$-stable Higgs subsheaves of equal slope.
        \end{itemize}
    \end{defn}
    With these notions out of the way, all the well-known results regarding semistability of ordinary sheaves, see for example \cite{huybrechtsLehn_ModuliOfSheaves}, are also valid for (log-) Higgs sheaves. For example, given a (log-) Higgs sheaf $(\mathcal{E}, \theta)$, there exists a unique saturated subsheaf $\mathcal{E}_{\alpha}\subseteq (\mathcal{E}, \theta)$, called the \emph{maximal $\alpha$-destabilising Higgs subsheaf}, such that 
    \begin{itemize}
        \item[(1)] $\mu_\alpha(\mathcal{F})\leq \mu_\alpha(\mathcal{E}_{\alpha})$ for all Higgs subsheaves $\mathcal{F}\subseteq (\mathcal{E}, \theta)$.
        \item[(2)] If $\mu_\alpha(\mathcal{F})=\mu_\alpha(\mathcal{E}_{\alpha})$ for some Higgs subsheaf $\mathcal{F}\subseteq (\mathcal{E}, \theta)$, then $\mathcal{F}\subseteq \mathcal{E}_\alpha$.
    \end{itemize}
    Moreover, there exists a unique filtration
    \[
    0 = \mathcal{E}_0 \subsetneq \mathcal{E}_1 \subsetneq \mathcal{E}_2 \subsetneq \ldots \subsetneq \mathcal{E}_\ell
    \]
    by saturated Higgs subsheaves, called the \emph{$\alpha$-Harder-Narasimhan filtration}, such that
    \begin{itemize}
        \item[(1)] $\mathcal{G}_i := (\mathcal{E}_i/\mathcal{E}_{i-1})^{\vee\vee}$ is $\alpha$-semistable for any $i=1, \ldots, \ell$.
        \item[(2)] $\mu_\alpha(\mathcal{G}_1) > \mu_\alpha(\mathcal{G}_2) > \ldots > \mu_\alpha(\mathcal{G}_\ell)$. 
    \end{itemize}
    \begin{rem}
    \label{rem-semistability-tensor-product}
        Let $H_1, \ldots, H_{n-1}$ be ample divisors and set $\alpha:= H_1\cdots H_{n-1}$. If $(\mathcal{E}, \theta)$ is $\alpha$-semistable, then so is $(\mathcal{E}\mathrm{nd}(\mathcal{E}), \theta)$. Indeed, by \cite[Thm.\ 7.2]{Langer_BogomolovsInequalityHiggsSheavesNormalVarieties}, we may assume that $X$ is a curve, in which case the assertion is well-known \cite[Co.\ 3.8]{Simpson_HiggsBundlesAndLocalSystems}. Similarly, if $(\mathcal{E}, \theta)$ is $\alpha$-polystable, then so is $(\mathcal{E}\mathrm{nd}(\mathcal{E}), \theta)$ \cite[Lem.\ 4.7]{GKPT_HarmonicMetricsUniformisation}.
    \end{rem}

\subsubsection{Adapted Reflexive Differentials and Simpson's construction}

    A \emph{pair} $(X, D)$ consists in a normal variety $X$ and a $\Q$-Weil divisor $D$ on $X$ with coefficients in $[0,1]\cap \Q$. A finite cover $\pi\colon Y \rightarrow X$ is called \emph{adapted} to $D$, if $\pi^*D$ is integral. Note that such covers always exist \cite[Prop. 2.38]{ClaudonKebekusTaji_GenericSemipositivity}. 

    Given an adapted cover $\pi\colon Y \rightarrow X$, one can define a reflexive sheaf $\Omega_{(X, D, \pi)}^{[1]}$ on $Y$, called the sheaf of \emph{adapted reflexive differentials}, as in \cite[Def.\ 3.10]{ClaudonKebekusTaji_GenericSemipositivity}. It has the following properties:
    \begin{itemize}
        \item[(1)] There exist natural inclusions $\pi^{[*]}\Omega_X^{[1]}\subseteq \Omega_{(X, D, \pi)}^{[1]} \subseteq \Omega_Y^{[*]}(\log D_Y)$, where $D_Y := ( \pi^*D)_{\mathrm{red}}$, see \cite[Rmk. 3.11.1]{ClaudonKebekusTaji_GenericSemipositivity}.
        \item[(2)] $\det(\Omega_{(X, D, \pi)}^{[1]}) = \mathcal{O}_Y(\pi^*(K_X + D))$, see \cite[Rmk. 3.11.5]{ClaudonKebekusTaji_GenericSemipositivity}. 
    \end{itemize}
    Conceptually, one should think of $\Omega_{(X, D, \pi)}^{[1]}$ as the reflexive pull back of a hypothetical sheaf $\Omega_X^{[1]}(\log D)$ on $X$. This can be made precise using the framework of orbi-sheaves, see \cites{GuenanciaTaji_SemistabilityLogCotangentSheaf,Mul_Thesis}.

    \begin{con}(Simpson \cite{Simpson_VHSandUniformisation})
    \label{con-simpson-cotangent-sheaf}
    
    \noindent
        Let $(X, D)$ be a pair and let $\pi\colon Y\rightarrow X$ be an adapted cover. Then the morphism
        \[
        \theta_{\Simp} \colon \left(\Omega_{(X, D, \pi)}^{[1]}\oplus \mathcal{O}_Y\right) \otimes \mathcal{T}_Y(-\log D_Y) \rightarrow \Omega_{(X, D, \pi)}^{[1]}\oplus \mathcal{O}_Y, \quad (\eta, f)\otimes v \mapsto (0, \eta(v))
        \]
        defines a natural log-Higgs sheaf structure on $\Omega_{(X, D, \pi)}^{[1]}\oplus \mathcal{O}_Y$.
    \end{con}
    Now, exactly as in \cite[Rmk.\ after Quest.\ 5.4]{Langer_NoteOnBogomolovInstability}, see also \cite[Prop.\ 2.8]{IMM_3c2=c1^2}, one easily verifies the following:
    \begin{prop}
    \label{Prop-Properties-of-Simpsons-Higgs-Sheaf}
        Let $(X, D)$ be a projective pair of dimension $n$. Let $L_1, \ldots, L_{n-1}$ be nef divisors on $X$ such that $\alpha:= L_1\cdots L_{n-1}\neq 0$. Let $\mathcal{F}\subseteq \Omega_{(X, D, \pi)}$ be a saturated subsheaf. Then the maximal $\alpha$-destabilising Higgs subsheaf of 
        \[
        (\mathcal{F} \oplus \mathcal{O}_Y, \theta_{\Simp})
        \]
        is of the form $\mathcal{M} \oplus \mathcal{O}_X$ for some saturated subsheaf $\mathcal{M} \subseteq \mathcal{F}$.
        In particular, if $\mathcal{F}$ is $\alpha$-semistable and if $c_1(\mathcal{F})\cdot \alpha >0$, then $(\mathcal{F} \oplus \mathcal{O}_X, \theta)$ is $\alpha$-stable.
    \end{prop}

\subsubsection{Numerically Projectively Flat Sheaves}
    Throughout this subsection, let $X$ be a log terminal projective variety of dimension $n$ and let $H_1, \ldots, H_{n-2}$ be ample divisors on $X$.
    \begin{defn}
    \label{def-Higgs-num-flat}
         A reflexive Higgs sheaf $(\mathcal{E}, \theta)$ on $X$ of rank $r$ is said to be 
        \begin{itemize}
            \item[(1)] \emph{numerically Higgs flat} if it is $(H_1\cdots H_{n-1})$-semistable and 
            $$
            c_1(\mathcal{E})\cdot H_1\cdots H_{n-1} = c_2(\mathcal{E})\cdot H_1\cdots H_{n-2} = 0.
            $$
            \item[(2)] \emph{numerically projectively Higgs flat} if it is $(H_1\cdots H_{n-1})$-semistable and
            $$
            \Delta(\mathcal{E})\cdot H_1 \cdots H_{n-2} = \Big(2rc_2(\mathcal{E}) - (r-1)c_1(\mathcal{E})^2\Big)\cdot H_1\cdots H_{n-2} = 0.
            $$
        \end{itemize}
    \end{defn}
    \begin{ex}
    \label{ex-num=proj-Higgs-flat-implies-End-Higgs-flat}
        By Remark \ref{rem-semistability-tensor-product} and \cite[Lem.\ 3.18]{GKPT_MY_Inequality_Uniformisation_of_Canonical_models}, if a reflexive Higgs sheaf $(\mathcal{E}, \theta)$ is numerically projectively Higgs flat, then $(\mathcal{E}\mathrm{nd}(\mathcal{E}), \theta)$ is numerically Higgs flat.
        In this case, if $(\mathcal{E}, \theta)$ is polystable, then so is $(\mathcal{E}\mathrm{nd}(\mathcal{E}), \theta)$ \cite[Lem.\ 4.7]{GKPT_HarmonicMetricsUniformisation}.
    \end{ex}
    If $(\mathcal{E}, \theta)$ is numerically Higgs flat, then it is semistable with respect to any multipolarisation and satisfies $c_1(\mathcal{E}) \equiv 0$ and $c_2(\mathcal{E}) \equiv 0$, see \cite[Fact 6.5]{GKPT_HarmonicMetricsUniformisation}.
    In this case, if $(\mathcal{E}, \theta)$ is (poly)stable with respect to one multipolarisation, then it is (poly)stable with respect to any multipolarisation \cite[Fact 6.2]{GKPT_HarmonicMetricsUniformisation}. 
    In effect, we simply say that a numerically projectively flat Higgs sheaf is (poly)stable when we mean that it is (poly)stable with respect to some or, equivalently, any multipolarisation on $X$.

    \begin{thm}\emph{(Non-Abelian Hodge Correspondence on klt spaces \cite{GKPT_HarmonicMetricsUniformisation})}
    \label{thm-Non-Abelian-Hodge-Correspondence}

        \noindent
        Let $X$ be a maximally quasi-\'etale log terminal projective variety. Then 
        any numerically Higgs flat reflexive sheaf $(\mathcal{E}, \theta)$ on $X$ is locally free. Moreover, there exists a natural equivalence of categories 
        \begin{align*}
        \begin{split}
            \big\{\textmd{linear representations of } \pi_1(X) \big\} & \cong \big\{ \textmd{numerically Higgs flat vector bundles on } X \big\}
        \end{split}
        \end{align*}
        with the property that a representation $\rho\colon \pi_1(X)\rightarrow \mathrm{GL}_r(\C)$ is irreducible (completely reducible) if and only if the associated Higgs vector bundle $(\mathcal{E}, \theta)$ is (poly)stable.
    \end{thm}
    \begin{proof}
        In general, by \cite[Thm.\ 7.12]{Langer_SimpsonCorrespondencePositiveCHar}, if $(\mathcal{E}, \theta)$ is a numerically Higgs flat sheaf on a log terminal projective variety $X$, then $\mathcal{E}|_{X_{\mathrm{reg}}}$ is locally free. With this is mind, the correspondence is established in \cite[Thm.\ 6.6]{GKPT_HarmonicMetricsUniformisation}.
    \end{proof}

    The following is a special instance of a very general principle, see \cite{Zuo_NegativityKernelHiggsField, PopaWu_WeakPositivityHodgeModules}:
    \begin{prop}
        \label{prop-kernel-Higgs-field-seminegative}
        Let $(\mathcal{E}, \theta)$ be a Higgs numerically flat vector bundle on a smooth projective variety $X$. If $\mathcal{F}\subseteq (\mathcal{E}, \theta)$ is a sub vector bundle such that $\theta(\mathcal{F}\otimes \mathcal{T}_X) = 0$, then $\mathcal{F}^\vee$ is nef.
    \end{prop}
    \begin{proof}
      By \cite{Simpson_VHSandUniformisation}, $(\mathcal{E}, \theta)$ admits an \emph{harmonic metric} $h$ and a well-known computation shows that, because $\vartheta(\mathcal{F}) = 0$, the Chern curvature tensor of the metric $h|_{\mathcal{F}}$ is semi-negative in the sense of Griffiths; in particular, $\mathcal{F}^\vee$ is nef \cite[Thm.\ 1.12]{DPS_ManifoldsWithNefTangentBundle}. See for example \cite[Rem.\ 3.2.11]{Mul_Thesis} for details.  
    \end{proof}
    In particular, we recover the well-known fact that if $\mathcal{E}$ is a numerically flat vector bundle on $X$, then $\mathcal{E}$ and $\mathcal{E}^\vee$ are nef, c.f.\ \cite{DPS_ManifoldsWithNefTangentBundle}.

    \subsection{Chern class inequalities for reflexive sheaves}

    The following result has been known at least since the fundamental work of Miyaoka \cite{miyaoka_MiyaokaInequality}; to the authors knowledge it was first stated explicitly, for smooth varieties, in \cite{Langer_SemistableSheavesInPositiveCharacteristic}; see \cite[Thm.\ 2.4.6]{Mul_Thesis} for a detailled proof in the generality we require.
    \begin{thm}\emph{(Miyaoka \cite{miyaoka_MiyaokaInequality})}

        \noindent
        Let $X$ be a normal projective variety of dimension $n$ with quotient singularities in codimension two and let $L_1, \ldots, L_{n-2}$ be nef divisors on $X$.
        Let $(\mathcal{E}, \theta)$ be a reflexive Higgs sheaf of rank $r$ on $X$.
        Assume that $(\mathcal{E}, \theta)$ is generically semipositive and that $\det(\mathcal{E})$ is a nef $\Q$-line bundle.
        If $c_1(\mathcal{E})^2\cdot L_1\cdots L_{n-2} \neq 0$, then
        \label{thm-Basic-Chern-Class-Inequality-Higgs-Sheaves}
        \begin{align}
           2 c_2(\mathcal{E}) \cdot L_{1}\cdots L_{n-2}
           \geq \Big( 
           c_1(\mathcal{E})^2\cdot L_{1}\cdots L_{n-2}
           - \mu^{\max}_\alpha(\mathcal{E}, \theta)
           \Big),
           \label{eq-Basic-Chern-Class-Inequality-Higgs-Sheaves-1}
        \end{align}
        where $\alpha = c_1(\mathcal{E})\cdot L_1\cdots L_{n-2}$. In case the equality holds in \eqref{eq-Basic-Chern-Class-Inequality-Higgs-Sheaves-1} then the $\alpha$-Harder--Narasimhan filtration 
        $$
        0 = \mathcal{E}_0 \subsetneq \mathcal{E}_1 \subseteq \mathcal{E}_2 = \mathcal{E}
        $$ 
        of $(\mathcal{E}, \theta)$ has at most two terms. Furthermore, denoting $\mathcal{F} := \mathcal{E}_1 $ and $\mathcal{G} := (\mathcal{E}/\mathcal{F})^{**}$, the following numerical conditions are satisfied for any Weil $\Q$-divisor $L$ on $X$: 
        \begin{gather}
            \Delta(\mathcal{F})\cdot L_1\cdots L_{n-2} = 
            \Delta(\mathcal{G})\cdot L_1\cdots L_{n-2} = 0, 
            \label{eq-Basic-Chern-Class-Inequality-Higgs-Sheaves-2}\\
            \big(c_1(\mathcal{E}) - c_1(\mathcal{F})\big)\cdot L \cdot L_1\cdots L_{n-2} 
            = 
            c_1(\mathcal{G})\cdot L\cdot L_1\cdots L_{n-2} = 0, 
            \label{eq-Basic-Chern-Class-Inequality-Higgs-Sheaves-3}\\
            \big(c_2(\mathcal{E}) - c_2(\mathcal{F}) - c_2(\mathcal{G}) - c_1(\mathcal{F})c_1(\mathcal{G})\big)\cdot L_1\cdots L_{n-2} = 0.\label{eq-Basic-Chern-Class-Inequality-Higgs-Sheaves-4}
        \end{gather}
    \end{thm}
    In the remaining case $c_1(\mathcal{E})^2\cdot L_1\cdots L_{n-2} = 0$ we have the following result, which is proved in a similar spirit.
    \begin{thm}\emph{(Miyaoka \cite{miyaoka_MiyaokaInequality}, Ou \cite{Ou_GenericNefnessTangentBundle}, Iwai--Matsumura \cite{IwaiMatsumura_ManifoldsWithC2=0})}
        \label{thm-Basic-Chern-Class-Inequality-Higgs-Sheaves-v=1}

        \noindent
        Let $X$ be a normale projective variety of dimension $n$ and with quotient singularities in codimension two. Let $L_1, \ldots, L_{n-2}$ be nef divisors on $X$.
        Let $(\mathcal{E}, \theta)$ be a reflexive Higgs sheaf on $X$.
        Assume that $(\mathcal{E}, \theta)$ is generically semipositive and that $\det(\mathcal{E})$ is a nef $\Q$-line bundle.
        If $c_1(\mathcal{E})^2\cdot L_1\cdots L_{n-2} = 0$, then
        \begin{align}
           c_2(\mathcal{E})\cdot L_1\cdots L_{n-2} \geq 0.
           \label{eq-Basic-Chern-Class-Inequality-Higgs-Sheaves-v=1}
        \end{align}
        If equality holds in \eqref{eq-Basic-Chern-Class-Inequality-Higgs-Sheaves-v=1}, then one of the following two assertions holds true. Set $\alpha_H := H\cdot L_1\cdots L_{n-2}$ for some ample divisor $H$ on $X$; then either $c_1(\mathcal{E})\cdot L_1\cdots L_{n-2} \equiv 0$ and $(\mathcal{E}, \theta)$ is $\alpha_H$-semistable. Otherwise,
        $c_1(\mathcal{E})\cdot L_1\cdots L_{n-2} \neq 0$ and the $\alpha_H$-Harder--Narasimhan filtration
            $$
            0 = \mathcal{E}_0 \subsetneq \mathcal{E}_1 \subsetneq \ldots \subsetneq \mathcal{E}_\ell = \mathcal{E}
            $$
        is independent of $H$. Moreover, there exist rational numbers $\lambda_i$ such that the graded quotients $\mathcal{G}_i := (\mathcal{E}_i/\mathcal{E}_{i-1})^{\vee\vee}$ satisfy for all $i=1,\ldots, \ell$ and all Weil $\Q$-divisors $L$ on $X$ that
            \begin{gather}
            c_2(\mathcal{G}_i)\cdot L_1\cdots L_{n-2} = c_1(\mathcal{G}_i)^2\cdot L_1\cdots L_{n-2}=0, 
            \label{eq-Basic-Chern-Class-Inequality-Higgs-Sheaves-v=1-2}\\
            c_1(\mathcal{G}_i)\cdot L\cdot L_1\cdots L_{n-2} = \lambda_i \ c_1(\mathcal{E})\cdot L\cdot L_1\cdots L_{n-2},
            \label{eq-Basic-Chern-Class-Inequality-Higgs-Sheaves-v=1-4}\\
            \big(c_2(\mathcal{E}_{i}) - c_2(\mathcal{E}_{i-1}) -  c_2(\mathcal{G}_{i}) -  c_1(\mathcal{E}_{i-1})c_1(\mathcal{G}_{i})\big)\cdot L_1\cdots L_{n-2} = 0.
            \label{eq-Basic-Chern-Class-Inequality-Higgs-Sheaves-v=1-3}
        \end{gather}
    \end{thm}
    Again, a detailled argument in the generality we require may be found in \cite[Thm.\ 2.4.7]{Mul_Thesis}.

    \subsubsection{Reflexive sheaves with vanishing discriminant}

    The following is a variant of \cite[Thm.\ 3.1]{Langer_NoteOnBogomolovInstability}; see \cite[Thm.\ 2.4.10, Cor.\ 2.4.11]{Mul_Thesis} for a detailled proof:
    \begin{lem}
    \label{lem-bundle-vanishing-dicriminant-universally-stable}
       Let $X$ be a normal projective variety of dimension $n$ and with quotient singularities in codimension two, 
       let $1\leq m \leq n$ be an integer 
       and let $L$ be a nef divisor on $X$ such that $L^m\not\equiv 0$. 
       Let $(\mathcal{E}, \theta)$ be a reflexive Higgs sheaf on $X$. 
       Assume that $(\mathcal{E}, \theta)$ is $(L^{m-1}\cdot H_1\cdots H_{n-m})$-(semi)stable 
       for some fixed choice of ample divisors $H_1, \ldots, H_{n-m}$ and that 
       \[
       \Delta(\mathcal{E})\cdot L^m \cdot H_1 \cdots H_{n-m} = 0.
       \]
       Then $(\mathcal{E}, \theta)$ is $(L^{m-2}\cdot A_1 \cdots A_{n-m-1})$-(semi)stable for any ample divisors $A_1, \ldots, A_{n-m-1}$ on $X$.
    \end{lem}

   \section{Inequalities of Miyaoka-Yau-type}
    \label{sec-twisted-semistability-cotangent-bundles}
    
    In this section, we investigate the following conjecture which was  posed as a question by Langer \cite[Quest.\ 5.4]{Langer_NoteOnBogomolovInstability} and as a conjecture by Chan--Leung \cite{CL_ConjectureMiyaoaYauTypeInequalities}:
    \begin{conj}
        If $X$ is a log terminal projective variety of dimension $n$ such that $K_X$ is nef, then
        \[
        \Big(2(m+1)c_2(X) - m c_1(X)^2\Big)\cdot K_X^{m-2}\cdot H_1\cdots H_{n-m} \geq 0
        \tag{$\mathrm{MY}_\mathrm{m}$}
        \label{conj-Miyaoka-Type-Inequality}
        \]
        for any ample divisors $H_1, \ldots, H_{n-m}$ on $X$ and any integer $2\leq m \leq n$.
    \end{conj}
    There are several cases of interest in which Conjecture \eqref{conj-Miyaoka-Type-Inequality} is classically known to hold. Indeed, by work of Miyaoka \cite{miyaoka_MiyaokaInequality}, Conjecture \eqref{conj-Miyaoka-Type-Inequality} holds when $m\in \{0,1,2\}$. Moreover, by Guenancia--Taji \cite{GuenanciaTaji_SemistabilityLogCotangentSheaf}, building on earlier work of Yau \cite{Yau_KEMetrics} and Tsuji \cite{Tsuji_InequalityChernNumbers}, Conjecture \eqref{conj-Miyaoka-Type-Inequality} holds when $m = \dim X$. In particular, Conjecture \eqref{conj-Miyaoka-Type-Inequality} holds unconditionally if $\dim X \leq 3$.

    The goal of this section is to confirm Conjecture \eqref{conj-Miyaoka-Type-Inequality} in several new instances:
    \begin{itemize}
        \item[(1)] Under the assumption that the polarisation $[H_1] = \ldots = [H_{n-m}] =: [H] \in N^1(X)_{\R}$ is chosen sufficiently close to $[K_X]$, see Subsection \ref{ssec-twisted-semistabnility-polarisations-close-to-KX}.
        \item[(2)] Under the assumption that $K_X$ is semiample and the Iitaka fibration $f\colon X \rightarrow X_{\mathrm{can}}$ is a smooth Abelian fibration in codimension one over $X_{\mathrm{can}}$, see Subsection \ref{ssec-Twisted-semistability-abelian-fibrations}.
    \end{itemize}
    Let us point out that (1) will play a key role in our proof of the Inequality \eqref{INTRO-eq-Chern-class-inequality} in Section \ref{sec-Higher-Inequalities}.
    The key insight that allows us to confirm Conjecture \eqref{conj-Miyaoka-Type-Inequality} in cases (1) and (2) is the realisation that, following ideas of Miyaoka \cite{miyaoka_MiyaokaInequality} and Langer \cite{Langer_NoteOnBogomolovInstability}, Conjecture \eqref{conj-Miyaoka-Type-Inequality} can be reduced to the following question:
    \begin{conj}
    \label{conj:twisted-semistability-cotangent-sheaf}
    Let $X$ be a projective log canonical variety of dimension $n$ such that $K_X$ is nef. Let $0\leq m \leq n$ be an integer and let $H_1, \ldots, H_{n-m}$ be ample divisors on $X$. We say that \emph{Conjecture $(\mathrm{S}_{\mathrm{m}})$ holds for $H_1, \ldots, H_{n-m}$} if
        \begin{align}
            \frac{c_1(\mathcal{F})\cdot K_X^{m-1}\cdot H_1\cdots H_{n-m}}{\rk\mathcal{F}}  \
            \leq \ \frac{K_X^{m}\cdot H_1\cdots H_{n-m}}{m} 
            \tag{$\mathrm{S}_\mathrm{m}$}
            \label{eq-twisted-semistability-cotangent-bundles}
        \end{align}
    for any subsheaf $\mathcal{F} \subseteq \Omega_X^{[1]}$.
    \end{conj}
    The key advantage of Conjecture \eqref{eq-twisted-semistability-cotangent-bundles} over Conjecture \eqref{conj-Miyaoka-Type-Inequality} is that it only involves the slopes of subsheaves of $\Omega_X^{[1]}$, which makes it accessible to the techniques discussed in Subsection \ref{ssec-Higgs-Sheaves-Semistability}.

    \subsection{A conjecture on twisted semistability of the cotangent bundle}
    \label{ssec-Twisted-Semistability-Cotangent-Bundle}

    In this subsection, we collect some easy observations concerning Conjecture \eqref{eq-twisted-semistability-cotangent-bundles}. In particular, we demonstrate that Conjecture \eqref{eq-twisted-semistability-cotangent-bundles} implies Conjecture \eqref{conj-Miyaoka-Type-Inequality}, see Theorem \ref{thm-Miyaoka-Type_inequality}. Let us start by collecting the cases in which Conjecture \eqref{eq-twisted-semistability-cotangent-bundles} is known:
    \begin{itemize}
        \item[(1)] Conjecture ($\mathrm{S}_{\mathrm{n}}$) has been verified by Guenancia and Taji \cite[Thm.\ C]{GuenanciaTaji_SemistabilityLogCotangentSheaf}.
        \item[(2)] By the generic semispositivity theorem of Miyaoka \cite{miyaoka_MiyaokaInequality} and Campana--P\u{a}un \cite[Thm.\ 1.3]{CampanaPaun_GenericSemipositivityCotangentBundle}, for any subsheaf $\mathcal{F}\subseteq \Omega_{X}^{[1]}$it holds that
        \[
        c_1(\mathcal{F})\cdot K_X^{m-1}\cdot H_1\cdots H_{n-m}
            \ \leq \ K_X^{m}\cdot H_1\cdots H_{n-m}.
        \]
        We deduce that \eqref{eq-twisted-semistability-cotangent-bundles} holds as soon as the rank $r$ of the maximal $(K_X^{m-1}\cdot H_1\cdots H_{n-m})$-destabilising subsheaf of $\Omega_X^{[1]}$ satisfies that $r\geq m$. In particular, ($\mathrm{S}_1$) always holds.
    \end{itemize}
    \begin{rem}
    \label{rem-twisted-semistability-cotangent-bundles}
        While we formulate Conjecture \eqref{eq-twisted-semistability-cotangent-bundles} independently of the value of the numerical dimension $\nu := \nu(X)$, by far the most interesting case is the one when $m = \nu$. Indeed, Conjecture ($\mathrm{S}_{\nu + 1}$) is equivalent to requiring that
        \begin{align*}
            c_1(\mathcal{F})\cdot K_X^{\nu}\cdot H_1\cdots H_{n-\nu -1}
            \ \leq \   K_X^{\nu + 1}\cdot H_1\cdots H_{n-\nu -1}
            =0,
        \end{align*}
        which follows from \cite[Thm.\ 1.3]{CampanaPaun_GenericSemipositivityCotangentBundle}. Moreover, when $m > \nu + 1$ Conjecture ($\mathrm{S}_{\mathrm{m}}$) is a tautology (as the terms on either side of the inequality clearly vanish). On the other hand, the validity of Conjecture ($\mathrm{S}_{\nu}$) for all ample divisors $H_1, \ldots, H_{n-\nu}$ on $X$ would immediately yield the truth of Conjecture ($\mathrm{S}_{\mathrm{m}}$) for all $m < \nu$ by simply degenerating $[H_{m+1}],\ldots,[H_{\nu}]$ towards $[K_X]$ in $N^1(X)_\R$.
    \end{rem}

    In the following, we will also consider a natural variant of Conjecture \eqref{eq-twisted-semistability-cotangent-bundles}:
    \begin{conj}
    \label{conj-twisted-semistability-Simpson-cotangent-Higgs-sheaf}
    Let $X$ be a projective log canonical variety of dimension $n$ such that $K_X$ is nef. Let $0\leq m \leq n$ be an integer and let $H_1, \ldots, H_{n-m}$ be ample divisors on $X$. We say that Conjecture \emph{(}$\mathrm{S}'_{\mathrm{m}}$\emph{)} holds for $H_1, \ldots, H_{n-m}$ if
        \begin{align}
            \frac{c_1(\mathcal{F})\cdot K_X^{m-1}\cdot H_1\cdots H_{n-m}}{\rk\mathcal{F}}  \
            \leq \ 
            \frac{K_X^{m}\cdot H_1\cdots H_{n-m}}{m+1}
            \tag{$\mathrm{S}'_\mathrm{m}$}
            \label{eq-twisted-semistability-Simpson-cotangent-bundles}
        \end{align}
    for any Higgs subsheaf $\mathcal{F} \subseteq (\Omega_X^{[1]}\oplus \mathcal{O}_X, \theta_{\Simp})$.
    \end{conj}

    Here, the Higgs sheaf $(\Omega_X^{[1]}\oplus \mathcal{O}_X, \theta_{\Simp})$ was defined in \ref{con-simpson-cotangent-sheaf}.
    \begin{prop}\emph{(Conjecture \eqref{eq-twisted-semistability-cotangent-bundles} implies Conjecture \eqref{eq-twisted-semistability-Simpson-cotangent-bundles})}
    \label{prop-semistability-cotangent-bundle-implies-ss-simpson}
    
    \noindent
       Let $X$ be a projective log canonical variety of dimension $n$ such that $K_X$ is nef. Let $0\leq m \leq n$ be an integer and let $H_1, \ldots, H_{n-m}$ be ample divisors on $X$. If Conjecture \eqref{eq-twisted-semistability-cotangent-bundles} holds for $H_1, \ldots, H_{n-m}$, then also Conjecture \eqref{eq-twisted-semistability-Simpson-cotangent-bundles} holds for $H_1, \ldots, H_{n-m}$.
    \end{prop}
    \begin{proof}
    By Proposition \ref{Prop-Properties-of-Simpsons-Higgs-Sheaf}, the maximal $\alpha := (K_X^{m-1}\cdot H_1\cdots H_{n-m})$-destabilising Higgs subsheaf 
    \[
    \mathcal{E} \subseteq \Big(\Omega_X^{[1]}\oplus \mathcal{O}_X, \theta_{\Simp}\Big)
    \]
    is of the form $\mathcal{E} = \mathcal{F} \oplus \mathcal{O}_X$ for some saturated subsheaf $\mathcal{F}\subseteq \Omega_X^{[1]}$. Let us denote $r := \rk\mathcal{F}$. We distinguish the following two cases:

    If $r\geq m$, then it follows from the generic semipositivity theorem \cite[Thm.\ 1.3]{CampanaPaun_GenericSemipositivityCotangentBundle}, that
    \[
    \frac{c_1(\mathcal{E})\cdot \alpha}{r+1} \
    = \ \frac{c_1(\mathcal{F})\cdot \alpha}{r+1} \
    \leq \ \frac{K_X\cdot \alpha}{r+1} \
    = \ \frac{m+1}{r+1}\cdot \frac{K_X\cdot \alpha}{m+1} \
    \leq \ \frac{K_X\cdot \alpha}{m+1},
    \]
    as required. Otherwise, in case $r\leq m$, we use Conjecture \eqref{eq-twisted-semistability-cotangent-bundles} to compute
    \[
    \frac{c_1(\mathcal{E})\cdot \alpha}{r+1} \
    = \ \frac{c_1(\mathcal{F})\cdot \alpha}{r+1} \
    \leq \ \frac{r}{r+1}\cdot \frac{K_X\cdot \alpha}{m} \
    = \ \frac{r\cdot (m+1)}{(r+1)\cdot m}\cdot \frac{K_X \cdot \alpha}{m+1} \
    \leq \ \frac{K_X \cdot \alpha}{m+1}.
    \]
    In any case, the Proposition is proved.
    \end{proof}

    Finally, we demonstrate that Conjecture \eqref{eq-twisted-semistability-Simpson-cotangent-bundles} implies Conjecture \eqref{conj-Miyaoka-Type-Inequality}:

    \begin{thm}\emph{(Conjecture \eqref{eq-twisted-semistability-Simpson-cotangent-bundles} implies Conjecture \eqref{conj-Miyaoka-Type-Inequality})}
    \label{thm-Miyaoka-Type_inequality}

    \noindent
        Let $X$ be a log terminal projective variety of dimension $n$ such that $K_X$ is nef. Let $2\leq m\leq n$ be an integer and assume that Conjecture \eqref{eq-twisted-semistability-Simpson-cotangent-bundles} holds for some ample divisors $H_1, \ldots, H_{n-m}$ on $X$.
        Then
        \begin{align}
            \Big(2(m+1)c_2(X) - m c_1^2(X)\Big)\cdot K_X^{m-2}\cdot H_1\cdots H_{n-m} \geq 0.
            \label{eq-Miyaoka-Type-Inequality}
        \end{align}
    \end{thm}
    \begin{proof}
        In case $\nu(X) \leq 1$, the validity of \eqref{eq-Miyaoka-Type-Inequality} was shown in \cite[Thm.\ 1.2]{IMM_3c2=c1^2}, see also \cite[Thm.\ 6.6]{miyaoka_MiyaokaInequality} for the case of terminal varieties.
        Now, let us assume that $\nu(X) \geq 2$.
        Let 
        \[
        (\mathcal{E}, \theta) := \Big(\Omega_X^{[1]}\oplus \mathcal{O}_X, \theta_{\Simp}\Big)
        \]
        be the Simpson cotangent Higgs sheaf and let us denote by
        $$
        0 = \mathcal{E}_0 \subsetneq \mathcal{E}_1 \subsetneq \ldots \subsetneq \mathcal{E}_\ell = \mathcal{E}
        $$
        its $\alpha := (K_X^{m-1}H_1\cdots H_{n-m})$-Harder--Narsasimhan filtration. Then $\mathcal{E}$ is generically semipositive \cite{CampanaPaun_GenericSemipositivityCotangentBundle} and $\det(\mathcal{E}) \cong \mathcal{O}_X(K_X)$ is a nef $\Q$-line bundle.
        By assumption,
        $$
        \mu_\alpha^{\mathrm{max}}(\mathcal{E}, \theta) 
        = \mu_{\alpha}(\mathcal{E}_1) 
        \leq \frac{K_X^m\cdot H_1\cdots H_{n-m}}{m+1}.
        $$
        An application of Theorem \ref{thm-Basic-Chern-Class-Inequality-Higgs-Sheaves} yields
        \begin{align*}
            \begin{split}
                2\cdot c_2(X)\cdot & K_X^{m-2}\cdot H_1\cdots H_{n-m} \\
        & \geq \ K_X^m\cdot H_1\cdots H_{n-m} - \frac{K_X^m\cdot H_1\cdots H_{n-m}}{m+1}.
            \end{split}
        \end{align*}
        A simple rearrangement of the terms on either side then gives the required estimate
        \[
        \Big(2(m+1)c_2(X) - m c_1^2(X)\Big)\cdot K_X^{m-2}\cdot H_1\cdots H_{n-m} \geq 0,
        \]
        proving the Theorem.
    \end{proof}
    \begin{rem}
    \label{rem-twisted-semistability-conjecture-pairs}
        All statements in this subsection can be formulated and proved more generally for projective klt pairs $(X, D)$ such that $K_X + D$ is nef, see \cite[Sec.\ 4.1]{Mul_Thesis}.
    \end{rem}

    \subsection{Twisted semistability for polarisations close to \texorpdfstring{$K_X$}{}}
    \label{ssec-twisted-semistabnility-polarisations-close-to-KX}
    In this subsection, we verify Conjecture \eqref{eq-twisted-semistability-cotangent-bundles} unconditionally for polarisations close to $[K_X] \in N^1(X)_{\R}$. The following result is due to Enoki \cite{Enoki_TwistedSemistability} in case $X$ is smooth; our proof is very similar:
    \begin{thm}\emph{(cf.\ Enoki \cite{Enoki_TwistedSemistability})}
    \label{thm-Enoki}
    
        \noindent
        Let $X$ be a projective log canonical variety of dimension $n$ such that $K_X$ is nef of numerical dimension $\nu := \nu(X)$. Let $H$ be an ample divisor on $X$. Then, for any rational number $t>0$ and any subsheaf $\mathcal{F} \subseteq \Omega_X^{[1]}$, it holds that
        \begin{align}
            \frac{c_1(\mathcal{F})\cdot (K_X + tH)^{n-1}}{\rk\mathcal{F}}\ 
            \leq \ \frac{(K_X + tH)^{n}}{n}.\label{eq-Enoki-1}
        \end{align}
        In particular,
        \begin{align}
           c_1(\mathcal{F})\cdot K_X^{\nu}\cdot H^{n-\nu-1} \leq 0.\label{eq-Enoki-2}
        \end{align}
        In case equality holds in \eqref{eq-Enoki-2}, then, additionally,
        \begin{align}
            \frac{c_1(\mathcal{F})\cdot K_X^{\nu-1}\cdot H^{n-\nu}}{\rk\mathcal{F}} \ 
            \leq \ \frac{K_X^{\nu}\cdot H^{n-\nu}}{\nu}.\label{eq-Enoki-3}
        \end{align}
    \end{thm}
    \begin{proof}
    Set $r:= \rk\mathcal{F}$ and fix a rational number $t>0$. Choose an integer $\ell \gg 1 $ such that $(\ell\cdot t) H$ is Cartier and very ample. Choose a general hypersurface $H' \in |(\ell\cdot t) H|$ and set $D' := \frac{1}{\ell} \cdot H'$. 
    Then the pair $(X, D')$ is log canonical. Moreover,
    \[
    K_X + D'\sim_\Q K_X + t\cdot H
    \]
    is ample. Then \eqref{eq-Enoki-1} follows from the the semistability of the logarithmic cotangent sheaf of a minimal, log canonical pair of general type \cite[Thm.\ C]{GuenanciaTaji_SemistabilityLogCotangentSheaf}. 

    To deduce \eqref{eq-Enoki-2} and \eqref{eq-Enoki-3} from \eqref{eq-Enoki-1}, we expand \eqref{eq-Enoki-1} as a polynomial in $t$:
    \begin{align*}
    \begin{split}
    \frac{1}{r} \sum_{i=0}^{n-1} \binom{n-1}{i} & \Big(c_1(\mathcal{F})\cdot K_X^{n-1-i}\cdot H^i\Big) \ t^i \
    \leq \ \frac{1}{n} \sum_{i=0}^{n} \binom{n}{i} \Big(K_X^{n-i} \cdot H^i\Big) \ t^i
    \end{split}
    \end{align*}
    In the limit $t\rightarrow 0$, it follows from \eqref{eq-inequality-polynomials} that
    \[
    \frac{1}{r} \cdot \binom{n-1}{n-\nu-1} \cdot c_1(\mathcal{F}) \cdot K_X^\nu \cdot H^{n - \nu -1} \leq 0,
    \]
    i.e.\ \eqref{eq-Enoki-2} holds. Here, we used that $(K_X + D)^{\nu + 1} \equiv 0$. Moreover, in case of equality, \eqref{eq-inequality-polynomials} yields an estimate between the next highest order terms in $t$; to wit:
    \[
    \frac{1}{r} \cdot \binom{n-1}{n-\nu} \cdot c_1(\mathcal{F})\cdot K_X^{\nu -1}\cdot  H^{n - \nu} \leq \frac{1}{n} \cdot \binom{n}{n-\nu} \cdot K_X^{\nu} \cdot H^{n - \nu}
    \]
    Since
    \[
    n \cdot \binom{n-1}{n-\nu} = \nu \cdot \binom{n}{n-\nu},
    \]
    we deduce the validity of \eqref{eq-Enoki-3}, thereby finishing the proof of the Theorem.
    \end{proof}
    \begin{cor}
    \label{cor-twisted-semistability-cotangent-bundle-asymptotic-nearKX}
        Let $X$ be a projective log canonical variety of dimension $n$ such that $K_X$ is nef of numerical dimension $\nu$. Let $H$ be an ample divisor on $X$. Then there exists a number $\varepsilon_0 >0$ such that for any rational number $0<\varepsilon < \varepsilon_0$ Conjecture \emph{(}\hyperref[eq-twisted-semistability-cotangent-bundles]{$\mathrm{S}_{\nu}$}\emph{)} holds for $(K_X + \varepsilon H)^{n-\nu}$. In other words,
        \begin{align*}
            \frac{c_1(\mathcal{F})\cdot K_X^{\nu-1}\cdot(K_X + \varepsilon H)^{n-\nu}}{\rk\mathcal{F}} \
            \leq \ \frac{K_X^{\nu}\cdot (K_X + \varepsilon H)^{n-\nu}}{\nu}
        \end{align*}
        for any subsheaf $\mathcal{F} \subseteq \Omega_X^{[1]}$ and any $0<\varepsilon < \varepsilon_0$.
    \end{cor}
    \begin{proof}
        For any $\varepsilon \in \Q_{>0}$ we set 
        $\alpha_{\varepsilon} := K_X^{\nu-1}\cdot (K_X + \varepsilon H)^{n-\nu}$ and we denote by $\mathcal{F}_{\varepsilon} \subseteq \Omega_X^{[1]}$ the maximal $\alpha_{\varepsilon}$-destabilising subsheaf. By \cite[Thm.\ 2.2.(1)]{miyaoka_MiyaokaInequality}, there exists $\varepsilon_0 > 0$ such that $\mathcal{F} := \mathcal{F}_\varepsilon$ is independent of $\varepsilon$ for $0<\varepsilon < \varepsilon_0$. Now,
        \begin{align*}
             \mu_{\alpha_{\varepsilon}}(\mathcal{F})
             = \mu^{\max}_{\alpha_{\varepsilon}}\Big(\Omega_X^{[1]}\Big) 
        \geq \mu_{\alpha_{\varepsilon}}\Big(\Omega_X^{[1]}\Big),
        \end{align*}
        or, in other words,
        \[
        \frac{c_1(\mathcal{F})\cdot K_X^{\nu-1}\cdot (K_X + \varepsilon H)^{n-\nu}}{\rk\mathcal{F}} 
        \geq  
        \frac{K_X^{\nu}\cdot (K_X + \varepsilon H)^{n-\nu}}{n}.
        \]
        for any $0<\varepsilon < \varepsilon_0$. Expanding both sides as polynomials in $\varepsilon$, and appealing again to \eqref{eq-inequality-polynomials}, we deduce that
        \[
        c_1(\mathcal{F})\cdot K_X^{\nu}\cdot H^{n-\nu-1} \geq 0.
        \]
        On the other hand, $c_1(\mathcal{F})\cdot K_X^{\nu}\cdot H^{n-\nu-1} \leq 0$ by Theorem \ref{thm-Enoki}. We conclude that
        \[
        c_1(\mathcal{F})\cdot K_X^{\nu}\cdot H^{n-\nu-1} = 0.
        \]
        Then, by Theorem \ref{thm-Enoki},
        \[
        \frac{c_1(\mathcal{F})\cdot K_X^{\nu-1}\cdot H^{n-\nu}}{\rk\mathcal{F}} \
            \leq \ \frac{K_X^{\nu}\cdot H^{n-\nu}}{\nu}.
        \]
        Finally, by \eqref{eq-inequality-polynomials} again, this implies that
        \[
        \frac{c_1(\mathcal{F})\cdot K_X^{\nu-1}\cdot (K_X + \varepsilon H)^{n-\nu}}{\rk\mathcal{F}} \
        \leq \ 
        \frac{K_X^{\nu}\cdot (K_X + \varepsilon H)^{n-\nu}}{\nu}.
        \]
        As $\mathcal{F} := \mathcal{F}_\varepsilon \subseteq \Omega_X^{[1]}$ was defined to be the subsheaf of maximal slope, this concludes the proof of the Corollary.
    \end{proof}
    \begin{ques}
        Can one infer from Corollary \ref{cor-twisted-semistability-cotangent-bundle-asymptotic-nearKX}, possibly using the results obtained in Xiao's thesis \cite[Thm.\ 2.0.76]{xiao_PhDThesis}, that Conjecture \eqref{eq-twisted-semistability-cotangent-bundles} holds true for $m = \nu(X)$, even when $[H_1],\ldots, [H_{n-m}]$ are possibly different, as long as they are sufficiently close to $[K_X]$ as elements in $N^1(X)_{\R}$?
    \end{ques}
    \begin{cor}
    \label{cor-Miyaoka-type-inequality-near-KX}
        Let $X$ be a log terminal projective variety of dimension $n$ such that $K_X$ is nef of numerical dimension $\nu$. Let $H$ be an ample divisor on $X$. Then there exists a number $\varepsilon_0 >0$ such that for any rational number $0<\varepsilon < \varepsilon_0$
        \[
        \Big(2(\nu+1)c_2(X) - \nu c_1(X)^2\Big) \cdot K_X^{\nu-2} \cdot (K_X + \varepsilon H)^{n-\nu} \geq 0.
        \]
    \end{cor}
    
    \begin{rem}
        Again, the results in this subsection can be generalised to the case of klt pairs, see \cite[Sec.\ 4.1]{Mul_Thesis}.
    \end{rem}

    \subsection{Twisted semistability for smooth Abelian fibrations}
    \label{ssec-Twisted-semistability-abelian-fibrations}

    In this subsection we verify Conjecture \eqref{eq-twisted-semistability-cotangent-bundles} under the additional assumption that $K_X$ is semiample and that the Iitaka fibration $f\colon X \rightarrow Y$ is a smooth Abelian fibration, see Remark \ref{rem--twisetd-semistability-cotangent-bundle-abelian-group-scheme}.

    \begin{lem}
    \label{lem-higher-order0inequalities-base-case}
        Let $X$ be a log terminal projective variety of dimension $n$ such that $K_X$ is nef with $\nu(X) =\nu$. If $c_2(X)\cdot K_X^{\nu -1} \equiv 0$,
        then the $(K_X^{\nu-1}\cdot H_1\cdots H_{n-\nu})$-Harder--Narasimhan filtration
        \[
        0 = \mathcal{E}_0 \subsetneq \mathcal{E}_1 \subsetneq \mathcal{E}_2 \subsetneq \ldots \subsetneq \mathcal{E}_\ell = \Omega_X^{[1]}
        \]
        of $\Omega_X^{[1]}$ is independent of the choice of ample divisors $H_1, \ldots, H_{n-\nu}$ on $X$; moreover, denoting $\mathcal{G}_i := (\mathcal{E}_i/\mathcal{E}_{i-1})^{\vee\vee}$, there exist rational numbers $\lambda_i \in \Q$ such that
        \begin{itemize}
            \item[\textmd{(1)}] $c_2(\mathcal{G}_i) \cdot K_X^{\nu-1} \equiv 0$.
            \item[\textmd{(2)}] $c_1(\mathcal{G}_i)\cdot K_X^{\nu-1} \equiv \lambda_i \cdot K_X^\nu$. 
        \end{itemize}
    \end{lem}
    \begin{proof}
        Choose ample divisors $H_1, \ldots, H_{n-\nu}$ on $X$. By Theorem \ref{thm-Basic-Chern-Class-Inequality-Higgs-Sheaves-v=1}, the $(K_X^{\nu-1}\cdot H_1\cdots H_{n-\nu})$-Harder--Narasimhan filtration
        \begin{equation}
            0 = \mathcal{E}_0 \subsetneq \mathcal{E}_1 \subsetneq \mathcal{E}_2 \subsetneq \ldots \subsetneq \mathcal{E}_\ell = \Omega_X^{[1]}
            \label{eq-HN-90}
        \end{equation}
        of $\Omega_X^{[1]}$  is independent of $H_1$; in particular, $\mathcal{G}_i := (\mathcal{E}_i/\mathcal{E}_{i-1})^{\vee\vee}$ is $(K_X^{\nu-1}\cdot H'_1\cdots H_{n-\nu})$-semistable for any ample divisor $H_1'$. Moreover, there exist rational numbers $\lambda_i \in \Q$ such that
        \begin{itemize}
            \item[\textmd{(1)}] $c_2(\mathcal{G}_i) \cdot K_X^{\nu-1} \cdot H_2\cdots H_{n-\nu}=0$.
            \item[\textmd{(2)}] $c_1(\mathcal{G}_i)\cdot K_X^{\nu-1} \cdot H_2\cdots H_{n-\nu} \equiv \lambda_i \cdot K_X^\nu\cdot H_2\cdots H_{n-\nu}$. 
        \end{itemize}
        Permuting the roles played by the $H_i$, we find that \eqref{eq-HN-90} is independent of the choice of  $H_1, \ldots, H_{n-\nu}$ and that, for any ample divisors $H_1', \ldots, H_{n-\nu}'$, the sheaves $\mathcal{G}_i$ are $(K_X^{\nu-1}\cdot H'_1\cdots H'_{n-\nu})$-semistable with $c_2(\mathcal{G}_i) \cdot K_X^{\nu-1} \cdot H'_2\cdots H'_{n-\nu}=0$ and $c_1(\mathcal{G}_i)\cdot K_X^{\nu-1} \cdot H'_2\cdots H'_{n-\nu} \equiv \lambda_i \cdot K_X^\nu\cdot H'_2\cdots H'_{n-\nu}$, as required.
    \end{proof}
    \begin{cor}
    \label{cor-proof-higher-order-inequalities-induction-start}
        Let $X$ be a log terminal projective variety of dimension $n$ such that $K_X$ is nef. 
        If $c_2(X)\cdot K_X^{\nu(X) -1} \equiv 0$,
        then Conjecture \eqref{eq-twisted-semistability-cotangent-bundles} holds on $X$. In other words,
        \begin{equation}
            \frac{c_1(\mathcal{F})\cdot K_X^{m-1}\cdot H_1\cdots H_{n-m}}{\rk\mathcal{F}}  \
            \leq \ \frac{K_X^{m}\cdot H_1\cdots H_{n-m}}{m} 
            \label{eq-twisetd-semistability-cotangent-bundle-abelian-group-scheme}
        \end{equation}
        for any integer $0\leq m \leq n$, any ample divisors $H_1, \ldots, H_{n-m}$ on $X$ and any coherent subsheaf $\mathcal{F}\subseteq \Omega_X^{[1]}$.
    \end{cor}
    \begin{proof}
        By Remark \ref{rem-twisted-semistability-cotangent-bundles}, it suffices to verify \eqref{eq-twisetd-semistability-cotangent-bundle-abelian-group-scheme} in case $m = \nu := \nu(X)$.
        Then, in the notation of Lemma \ref{lem-higher-order0inequalities-base-case}, it suffices to verify that
        \[
        \lambda_1\leq \frac{\rk\mathcal{E}_1}{\rk\mathcal{E}}.
        \]
        But indeed, to do so, simply choose $H_1 = \ldots = H_{n-\nu} = K_X +\varepsilon H$ for a sufficiently small rational number $\varepsilon >0$ and apply Corollary \ref{cor-twisted-semistability-cotangent-bundle-asymptotic-nearKX}.
    \end{proof}
    \begin{rem}
    \label{rem--twisetd-semistability-cotangent-bundle-abelian-group-scheme}
       Let $X$ be a log terminal projective variety of dimension $n$ such that $K_X$ is semiample with Iitaka fibration $f\colon X \rightarrow Y$. Then it follows from \cite{IMM_3c2=c1^2}, that the assumptions of Corollary \ref{cor-proof-higher-order-inequalities-induction-start} are satisfied, if (and essentially only if) $f$ is a \emph{smooth Abelian fibration in codimension one} over $Y$, i.e.\ there exists an open subset $U\subseteq Y$, with $\codim (Y\setminus U) \geq 2$, such that $f\colon V := f^{-1}(U) \rightarrow U$ is a smooth morphism with Abelian varieties as fibres.
    \end{rem}
    Note that, assuming $K_X$ to be semiample, an alternative proof of Corollary \ref{cor-proof-higher-order-inequalities-induction-start} is given in \cite[Sec.\ 4.3]{Mul_Thesis} by analysing the geometry of the Iitaka fibration. This leads to the following, admittedly very vague, question:
    \begin{ques}
        Let $X$ be a log terminal projective variety of dimension $n$ such that $K_X$ is semiample. What is the relation between the maximal $(K_X^{\nu-1}\cdot H)$-destabilising Higgs subsheaf of $(\Omega_X^{[1]}\oplus \mathcal{O}_X, \theta_{\Simp})$ and the Iitaka fibration of $X$? 
    \end{ques}

    \section{Higher order inequalities and proof of the Inequality \texorpdfstring{(\ref{INTRO-eq-Chern-class-inequality})}{}}
    \label{sec-Higher-Inequalities}

    This goal of this section is to prove the following result, see subsection \ref{ssec-higher-order-inequalities}:
    \begin{thm}
    \label{intro-thm-higher-inequalities-cycle-classes}
        Let $X$ be a log terminal projective variety of dimension $n$ such that $K_X$ is nef and set $\nu := \nu(X)$. If there exists an integer $k\geq 2$ such that
        \begin{equation}
             \Big(2(\nu+1)c_2(X) - \nu c_1(X)^2\Big)\cdot K_X^{k-1} \equiv 0,
             \label{intro-eq-higher-inequalities-assumption}
        \end{equation}
        then for all ample divisors $H_1, \ldots, H_{n-k}$ on $X$ it holds that
        \begin{equation}
            \Big(2(\nu+1)c_2(X) - \nu c_1(X)^2\Big)\cdot K_X^{k-2} \cdot H_1\cdot H_{n-k} \geq 0.
            \label{intro-eq-higher-inequalities}
        \end{equation}
    \end{thm}
    Here, recall that we denote $(2(\nu+1)c_2(X) - \nu c_1(X)^2)\cdot  K_X^{k-2} \equiv 0$ to mean that 
    \[
        \Big(2(\nu+1)c_2(X) - \nu c_1(X)^2\Big)\cdot  K_X^{k-2} \cdot L_1\cdots L_{n-k} = 0
    \]
    for all Cartier divisors $L_1, \ldots, L_{n-k}$ on $X$. Note that by Theorem \ref{intro-thm-higher-inequalities-cycle-classes}, \cite[Prop.\ 2.6]{IMM_3c2=c1^2} and descending induction on $k$, the identity \eqref{intro-eq-higher-inequalities-assumption} is equivalent to the existence of an ample divisor $H$ on $X$ such that
    \begin{equation}
             \Big(2(\nu+1)c_2(X) - \nu c_1(X)^2\Big)\cdot K_X^{j-2} \cdot H^{n-j} = 0, \qquad \forall j=k+1, \ldots, n.
             \label{intro-eq-higher-inequalities-assumption-alternative}
    \end{equation}
    As an immediate consequence of Theorem \ref{intro-thm-higher-inequalities-cycle-classes}, we deduce \eqref{INTRO-eq-Chern-class-inequality}, see Subsection \ref{ssec-proof-of-star}:
    \begin{cor}
    \label{cor-inequality-star}
       Let $X$ be a log terminal projective variety of dimension $n$ such that $K_X$ is nef, with $\nu := \nu(X)$. Then there exists a number $\varepsilon_0 >0$ such that
    \begin{equation*}
        \Big(2(\nu+1)c_2(X) - \nu c_1(X)^2\Big)\cdot (K_X + \varepsilon H)^{n-2} \geq 0, \qquad \forall\ 0 \leq \varepsilon < \varepsilon_0.
        \tag{$\star$}
    \end{equation*}
    \end{cor}
    \noindent
    Theorem \ref{intro-thm-higher-inequalities-cycle-classes} is new even in the case $\nu = n$ and, to the best of the author's knowledge, has not been conjectured or expected before. 
    \begin{rem}
        In general, one can not expect \eqref{intro-eq-higher-inequalities} to hold without assuming \eqref{intro-eq-higher-inequalities-assumption}. Indeed, let $S$ be a K3 surface, let $B = \B^2/\Lambda$ be a smooth ball quotient surface and let $C$ be a smooth projective curve of general type. Denote $X:= S\times B \times C$. Then $X$ is a smooth projective variety, $K_X$ is nef and $\nu(X) = 3$. Let $H_S, H_B$ and $H_C$ be ample $\Q$-divisors on $S, B$ and $C$, respectively, such that $H_S^2 = H_B^2 = 1$ and $\deg H_C = 1$. Then
        \[
        \Big(8c_2(X) - 3c_1(X)^2\Big) \cdot  pr_3^*H_C \cdot pr_1^*H_S^2 = 8c_2(B) - 3c_1(B)^2 < 0,
        \]
        while
        \[
        \Big(8c_2(X) - 3c_1(X)^2\Big) \cdot pr_2^*H_B^2 \cdot pr_3^*H_C = 8c_2(S) - 3c_1(S)^2 = 8c_2(S) > 0.
        \]
        In particular, we can find ample divisors $H_1, H_2, H_3$ on $X$ such that the number $ (2(\nu+1)c_2(X) - \nu c_1(X)^2) \cdot H_1\cdots H_{n-2}$ is positive, negative or zero.
    \end{rem}
    During the course of the proof of Theorem \ref{intro-thm-higher-inequalities-cycle-classes}, we will also verify the following lemma, which will be useful later in Sections \ref{sec-uniformisation-v>=2-Semiampleness} and \ref{sec-proof-of-thm-B}:
    \begin{lem}
    \label{lem-uniformisation-light}
        Let $X$ be a log terminal projective variety of dimension $n$ such that $K_X$ is nef. 
        Let $k$ be an integer satisfying $2\leq k\leq  \nu := \nu(X)$. 
        If
        \begin{equation*}
            \Big(2(\nu+1)c_2(X) - \nu c_1(X)^2\Big) \cdot K_X^{k-2} \equiv 0,
        \end{equation*}
        then there is a short exact sequence
        \[
        0 \rightarrow \mathcal{F} \rightarrow \Omega_X^{[1]} \rightarrow \mathcal{Q} \rightarrow 0
        \]
        of coherent sheaves on $X$ such that, for any ample divisors $H_1, \ldots, H_{n-k+1}$,
        \begin{itemize}
        \setlength\itemsep{0.2em}
            \item[\emph{(0)}] $\mathcal{F}$ is reflexive and $\mathcal{Q}$ is torsion-free; moreover, $\rk\mathcal{F} = \nu$.
            \item[\emph{(1)}] $(\mathcal{F}\oplus \mathcal{O}_X, \theta_{\Simp})$ is $(K_X^{k-2}\cdot H_1\cdots H_{n-k+1})$-stable.
            \item[\emph{(2)}] $\mathcal{Q}$ is $(K_X^{k-2}\cdot H_1\cdots H_{n-k+1})$-semistable.
            \item[\emph{(3)}] $c_1(\mathcal{F})\cdot K_X^{k-2} \equiv K_X^{k-1}$ and $c_2(\mathcal{F})\cdot K_X^{k-2} \equiv c_2(X)\cdot K_X^{k-2}$.
            \item[\emph{(4)}] $c_1(\mathcal{Q})\cdot K_X^{k-2} \equiv 0$ and $c_2(\mathcal{Q}^{\vee\vee})\cdot K_X^{k-2} \equiv 0$. 
        \end{itemize}
    \end{lem}

    \subsection{Proof of the higher order inequalities}
    \label{ssec-higher-order-inequalities}
    In this subsection, we will prove Theorem \ref{intro-thm-higher-inequalities-cycle-classes} and Lemma \ref{lem-uniformisation-light} simultaneously by descending induction on $k\geq 2$.

    We start by observing that in case $k\geq \nu(X)+1$ the assertion in Theorem \ref{intro-thm-higher-inequalities-cycle-classes} follows immediately from the inequality
    \begin{equation*}
            c_2(X)\cdot H_1\cdots H_{n-2} \geq 0,
    \end{equation*}
    which holds for any ample divisors $H_1, \ldots, H_{n-2}$ \cite[Thm.\ 1.2]{IMM_3c2=c1^2}, see also \cite[Thm.\ 6.6]{miyaoka_MiyaokaInequality} for the terminal case. Moreover, in case $k = \nu(X)$, the assertion of Theorem \ref{intro-thm-higher-inequalities-cycle-classes} follows immediately from Corollary \ref{cor-proof-higher-order-inequalities-induction-start}, see also
    Theorem \ref{thm-Miyaoka-Type_inequality} and Proposition \ref{prop-semistability-cotangent-bundle-implies-ss-simpson}.

    \begin{proof}[Proof of Theorem \ref{intro-thm-higher-inequalities-cycle-classes} and Lemma \ref{lem-uniformisation-light}]
    Fix a log terminal projective variety $X$ of dimension $n$ such that $K_X$ is nef. Set $\nu := \nu(X)$. We consider the following assertions:
    \begin{itemize}
        \item[($\mathrm{T}_{\mathrm{k}}$)] For any ample divisors $H_1, \ldots, H_{n-k}$ on $X$ and any Higgs subsheaf $\mathcal{E}\subsetneq (\Omega_X^{[1]}\oplus \mathcal{O}_X, \theta_{\Simp})$ it holds that
        \[
        \frac{c_1(\mathcal{E})\cdot K_X^{k-1}\cdot H_1\cdots H_{n-k}}{\rk\mathcal{E}}
        \ \leq \
        \frac{K_X^{k}\cdot H_1\cdots H_{n-k}}{\nu +1}.
        \]
        \item[($\mathrm{I}_{\mathrm{k}}$)] For any ample divisors $H_1, \ldots, H_{n-k}$ on $X$ it holds that
        \[
        \Big(2(\nu+1)c_2(X) - \nu c_1(X)^2\Big) \cdot K_X^{k-2} \cdot H_1 \cdots H_{n-k} \geq 0.
        \]

        \item[($\mathrm{E}_{\mathrm{k}}$)] It holds that $(2(\nu+1)c_2(X) - \nu c_1(X)^2) \cdot K_X^{k-2} \equiv 0$.
        
        \item[($\mathrm{F}_{\mathrm{k}}$)] There is a short exact sequence
        \[
        0 \rightarrow \mathcal{F} \rightarrow \Omega_X^{[1]} \rightarrow \mathcal{Q} \rightarrow 0
        \]
        satisfying properties (0)-(4) in Lemma \ref{lem-uniformisation-light}$_{\mathrm{k}}$.
    \end{itemize}

    In the following, we are going to show, by descending induction on $2\leq k\leq \nu$ that ($\mathrm{F}_{\mathrm{k+1}}$) $\Rightarrow$  ($\mathrm{T}_{\mathrm{k}}$) $\Rightarrow$  ($\mathrm{I}_{\mathrm{k}}$) and that ($\mathrm{T}_{\mathrm{k}}$) and ($\mathrm{E}_{\mathrm{k}}$) together imply ($\mathrm{F}_{\mathrm{k}}$). In particular, this shows Theorem \ref{intro-thm-higher-inequalities-cycle-classes} and Lemma \ref{lem-uniformisation-light}.

    To start the induction, observe that ($\mathrm{T}_{\nu}$) follows from Corollary \ref{cor-proof-higher-order-inequalities-induction-start}, see Proposition \ref{prop-semistability-cotangent-bundle-implies-ss-simpson}.

    Next, let us show that ($\mathrm{T}_{\mathrm{k}}$) $\Rightarrow$  ($\mathrm{I}_{\mathrm{k}}$). Let $\mathcal{E} \subseteq (\Omega_X^{[1]}\oplus\mathcal{O}_X, \theta_{\Simp})$ be the maximal $(K_X^{k-1}\cdot H_1\cdots H_{n-k})$-destabilising Higgs subsheaf. Recall, that $\mathcal{E} = \mathcal{F}\oplus \mathcal{O}_X$ for some saturated subsheaf $\mathcal{F}\subseteq \Omega_X^{[1]}$, see Proposition \ref{Prop-Properties-of-Simpsons-Higgs-Sheaf}. By Theorem \ref{thm-Basic-Chern-Class-Inequality-Higgs-Sheaves}, 
    \begin{align}
        \begin{split}
            2 c_2 & (X)\cdot K_X^{k-2} \cdot H_1\cdots H_{n-k} \\
            & \geq K_X^{k}\cdot  H_1\cdots H_{n-k}  - \mu^{\max}_{\alpha}\Big(\Omega_X^{[1]}\oplus \mathcal{O}_X, \theta_\Simp\Big) \\
           & = K_X^{k}\cdot  H_1\cdots H_{n-k} - \mu_\alpha(\mathcal{E}) \\
           & = \frac{\nu}{\nu+1} \cdot K_X^{k}\cdot  H_1\cdots H_{n-k},
        \end{split}
        \label{eq-computation-45}
    \end{align}
    where we used ($\mathrm{T}_{\mathrm{k}}$) in the final step. Rearranging the terms on either sides yields ($\mathrm{I}_{\mathrm{k}}$).
    Now, assume that the equality in \eqref{eq-computation-45} holds. Set $\mathcal{Q} := \Omega_X^{[1]}/\mathcal{F}$ and $\mathcal{G} := \mathcal{Q}^{\vee\vee}$. Theorem \ref{thm-Basic-Chern-Class-Inequality-Higgs-Sheaves} implies that
    \begin{gather}
     c_1(\mathcal{F})\cdot K_X^{k-2} \equiv K_X^{k-1}     
    \quad \mathrm{and} \quad
     c_1(\mathcal{Q})\cdot K_X^{k-2} \equiv 0,
     \label{eq-proof-higher-order-inequalities-1}
    \\
    \Delta(\mathcal{F}\oplus \mathcal{O}_X)\cdot K_X^{k-2} \equiv 0
    \quad \mathrm{and} \quad
    \Delta(\mathcal{Q})\cdot K_X^{k-2} \equiv 0,
     \label{eq-proof-higher-order-inequalities-2}
    \\
    c_2(\mathcal{F})\cdot K_X^{k-2} \equiv c_2(X)\cdot K_X^{k-2}
    \quad \mathrm{and} \quad
    c_2(\mathcal{Q})\cdot K_X^{k-2} \equiv 0.
     \label{eq-proof-higher-order-inequalities-3}
    \end{gather}
    In particular, by \eqref{eq-proof-higher-order-inequalities-1} and \eqref{eq-proof-higher-order-inequalities-3},  
    \[
    \Big(2(\nu+1) c_2(\mathcal{F}) - \nu c_1(\mathcal{F})^2\Big)\cdot K_X^{k-2} 
    \equiv
    \Big(2(\nu+1) c_2(X) - \nu c_1(X)^2\Big)\cdot K_X^{k-2} 
    \equiv 0.
    \]
    Comparing with \eqref{eq-proof-higher-order-inequalities-2}, we find that $\rk\mathcal{F} = \nu$. Moreover, from Lemma \ref{lem-bundle-vanishing-dicriminant-universally-stable} and \eqref{eq-proof-higher-order-inequalities-2} we infer that $(\mathcal{F}\oplus \mathcal{O}_X, \theta_\Simp)$ is stable with respect to $K_X^{k-2}\cdot H'_1\cdots H'_{n-k+1}$ for any ample divisors $H_1',\ldots, H_{n-k+1}'$. Furthermore, from \eqref{eq-proof-higher-order-inequalities-1} and the generic semipositivity of $\Omega_X^{[1]}$ \cite{CampanaPaun_GenericSemipositivityCotangentBundle}, we deduce that $\mathcal{Q}$ is semistable with respect to $K_X^{k-2}\cdot H'_1\cdots H'_{n-k+1}$. In total, ($\mathrm{F}_{\mathrm{k}}$) holds.

    In summary, we have verified that ($\mathrm{T}_{\mathrm{k}}$) $\Rightarrow$  ($\mathrm{I}_{\mathrm{k}}$) and that (($\mathrm{T}_{\mathrm{k}}$) $+$ ($\mathrm{E}_{\mathrm{k}}$)) $\Rightarrow$ ($\mathrm{F}_{\mathrm{k}}$). Thus, to finish the induction, and, hence, the proof of Theorem \ref{intro-thm-higher-inequalities-cycle-classes} and Lemma \ref{lem-uniformisation-light}, it remains to verify that, for $k< \nu$, ($\mathrm{F}_{\mathrm{k+1}}$) $\Rightarrow$  ($\mathrm{T}_{\mathrm{k}}$). But indeed, by ($\mathrm{F}_{\mathrm{k+1}}$), we have the short exact sequence
    \begin{equation}
        0 \rightarrow \mathcal{F}\oplus\mathcal{O}_X \rightarrow \Omega_X^{[1]}\oplus \mathcal{O}_X \rightarrow \mathcal{Q} \rightarrow 0.
        \label{eq-ses-79}
    \end{equation}
    In fact, this is the $\alpha := (K_X^{k-1}H_1\cdots H_{n-k})$-Harder--Narasimhan filtration of
    \[
    \Big(\Omega_X^{[1]}\oplus \mathcal{O}_X, \theta_\Simp\Big);
    \]
    indeed, both $(\mathcal{F}\oplus\mathcal{O}_X, \theta_\Simp)$ and $\mathcal{Q}^{\vee\vee}$ are $\alpha$-semistable by part (1) and (2) of Lemma \ref{lem-uniformisation-light}$_{\mathrm{k}}$. Moreover, by part (3) and (4) of Lemma \ref{lem-uniformisation-light}$_{\mathrm{k}}$,
    \[
    \mu_\alpha(\mathcal{F}\oplus \mathcal{O}_X) > 0 = \mu_\alpha(\mathcal{Q}).
    \]
    This shows that \eqref{eq-ses-79} is the $\alpha$-Harder--Narasimhan filtration of the Higgs sheaf $(\Omega_X^{[1]}\oplus \mathcal{O}_X, \theta_\Simp)$, as asserted. In particular, ($\mathrm{T}_{\mathrm{k}}$) holds.
    \end{proof}

    \subsection{Proof of the Inequality (\texorpdfstring{\ref{INTRO-eq-Chern-class-inequality}}{})}
    \label{ssec-proof-of-star}

    \begin{proof}[Proof of Corollary \ref{cor-inequality-star} \emph{(}= Inequality \eqref{INTRO-eq-Chern-class-inequality}\emph{)}]
        Let $X$ be a log terminal projective variety of dimension $n$ such that $K_X$ is nef of numerical dimension $\nu$. Then
        \begin{equation}
             \Big(2(\nu+1) c_2(X) - \nu c_1(X)^2\Big) \cdot K_X^{n-2} \geq 0.
             \label{eq-proof-of-star-1}
        \end{equation}
        Indeed, if $\nu = n$, then \eqref{eq-proof-of-star-1} is just \cite[Thm.\ 1.1]{GKPT_MY_Inequality_Uniformisation_of_Canonical_models}; otherwise, \eqref{eq-proof-of-star-1} is equivalent to
        \[
        c_2(X) \cdot K_X^{n-2} \geq 0,
        \]
        which is true, for example by \cite[Thm.\ 1.1]{GKPT_MY_Inequality_Uniformisation_of_Canonical_models}. With this in mind, Corollary \ref{cor-inequality-star} follows immediately from Theorem \ref{intro-thm-higher-inequalities-cycle-classes} by descending induction on 
        \[
        k := \min\left\{j \geq 2 \ | \ \Big(2(\nu+1)c_2(X) - \nu c_1(X)^2\Big)\cdot K_X^{j-2} \neq 0 \right\}.
        \]
    \end{proof}

    \begin{rem}
        \label{rem-Higher-Inequalities-klt-pairs}
        Once more, all results in this section can be generalised to the case of klt pairs; we leave it to the interested reader to verify the details. 
    \end{rem}

\section{On the semiampleness of the canonical divisor in Theorem \ref{INTRO-thm-statement-light}}
    \label{sec-uniformisation-v>=2-Semiampleness}

    In this section, we are concerned with the proof of the semiampleness statement in Theorem \ref{INTRO-thm-statement-light}. Taking into account the results obtained in Section \ref{sec-Higher-Inequalities}, see \eqref{intro-eq-higher-inequalities-assumption-alternative}, the statement reads as follows:
    \begin{thm}
    \label{thm-uniformisation-v=c-semiampleness-part}
        Let $X$ be a log terminal projective variety such that $K_X$ is nef. If
        \[
        2\big(\nu(X)+1\big) \cdot c_2(X) \equiv \nu(X) \cdot c_1(X)^2,
        \]
        then $K_X$ is semiample.
    \end{thm}
    The proof of Theorem \ref{thm-uniformisation-v=c-semiampleness-part} will take up the rest of this section. As a preparation, let us prove the following Lemma:
    \begin{lem}
    \label{lem-uniformisation-light-c=m=v}
        Let $X$ be a log terminal projective variety such that $K_X$ is nef of numerical dimension $\nu(X) \geq 2$. Assume that $X$ is maximally quasi-\'etale. If
        \[
        2(\nu+1)\cdot c_2(X) \equiv \nu \cdot c_1(X)^2,
        \]
        then $X$ is smooth, not uniruled and there exists a short exact sequence of vector bundles
        \[
        0\rightarrow \mathcal{F} \rightarrow \Omega_X^1 \rightarrow \mathcal{Q} \rightarrow 0
        \]
        on $X$ such that $\mathcal{Q}$ is numerically flat, polystable and $\rk\mathcal{Q} = n- \nu$. Moreover,  the Higgs bundle $(\mathcal{E}, \theta) := (\mathcal{F} \oplus \mathcal{O}_X, \theta_\Simp)$ is numerically projectively Higgs flat and stable.
    \end{lem}
    \begin{proof}
        By Lemma \ref{lem-uniformisation-light}, there exists a short exact sequence
        \begin{equation}
            0 \rightarrow \mathcal{F} \rightarrow \Omega_X^{[1]} \rightarrow \mathcal{Q} \rightarrow 0,
            \label{eq-3c2=c1^2-ses}
        \end{equation}
        where $\mathcal{F}$ is a reflexive coherent sheaf of rank $\nu$ and $\mathcal{Q}$ is torsion-free. Moreover, the Higgs sheaf $(\mathcal{E}, \theta) := (\mathcal{F}\oplus \mathcal{O}_X, \theta_\Simp)$ is stable with respect to any multipolarisation, while the sheaf $\mathcal{Q}$ is semistable with respect to any multipolarisation. Furthermore,
        \[
        c_2(\mathcal{F}) \equiv c_2(X),\quad
        c_2(\mathcal{Q}^{\vee\vee}) \equiv 0,\quad
        c_1(\mathcal{F}) \equiv K_X\quad \mathrm{and} \quad
        c_1(\mathcal{Q}) \equiv 0.
        \]
        In particular, $(\mathcal{E}, \theta)$ is numerically projectively Higgs flat, $\mathcal{Q}^{\vee\vee}$ is numerically flat and \eqref{eq-3c2=c1^2-ses} is Zariski locally split in codimension two, see \cite[Lem.\ 2.13]{IMM_3c2=c1^2}. By \cite[Thm.\ 1.4]{LuTaji_QuasiEtaleQuotientsAbelianVarieties}, the sheaf $\mathcal{Q}^{\vee\vee}$ is locally free. Moreover, by the argument in \cite[Cor.\ 2.11]{IMM_3c2=c1^2}, the same is true of $\mathcal{E}$ and $\mathcal{F}$. Then \cite[Lemma 9.9]{AD_FanoFoliations} shows that the sequence \eqref{eq-3c2=c1^2-ses} is exact and that all sheaves involved are locally free. By the resolution of the Lipman--Zariski conjecture for klt spaces, \cite[Thm.\ 6.1]{GKKP_DifferentialFOrmsLogCanonicalSpaces}, $X$ is smooth. Finally, since $K_X$ is nef, $X$ can not be uniruled \cite[Thm.\ IV.1.9, Ex.\ II.3.1.3]{Kollar_RationalCurvesOnAlgVarieties} and the polystability of $\mathcal{Q}$ follows from \cite[Lem.\ 2.1]{PereiraTouzet_FoliationsVanishingChernClasses}.
    \end{proof}
    We are now in position to prove Theorem \ref{thm-uniformisation-v=c-semiampleness-part}:
    \begin{proof}[Proof of Theorem \ref{thm-uniformisation-v=c-semiampleness-part}]

        We may assume that $X$ is maximally quasi-\'etale. According to Lemma \ref{lem-uniformisation-light-c=m=v}, $X$ is smooth and there exists a short exact sequence of vector bundles
        \begin{equation}
            0\rightarrow \mathcal{F} \rightarrow \Omega_X^1 \rightarrow \mathcal{Q} \rightarrow 0
            \label{eq-ses-123}
        \end{equation}
        on $X$ such that $\rk\mathcal{F} = \nu$ and such that $\mathcal{Q}$ is numerically flat and polystable. Let $\rho\colon \pi_1(X) \rightarrow \mathrm{GL}_{n-\nu}(\C)$ be the corresponding representation, see Theorem \ref{thm-Non-Abelian-Hodge-Correspondence}.
        Moreover,
        $(\mathcal{E}, \theta) := (\mathcal{F} \oplus \mathcal{O}_X, \theta_\Simp)$ is numerically projectively Higgs flat and stable. By a  version of Theorem \ref{thm-Non-Abelian-Hodge-Correspondence}, see \cite{PZZ_ProjectivelyFlatSimpsonCorrespondence}, \cite[Sect.\ 8]{Simpson_VHSandUniformisation}, to $(\mathcal{E}, \theta)$ we can associate an irreducible projective representation $\sigma\colon \pi_1(X) \rightarrow \mathrm{PGL}_{\nu +1}(\C)$. Concretely, by Example \ref{ex-num=proj-Higgs-flat-implies-End-Higgs-flat}, $(\mathcal{E}\mathrm{nd}(\mathcal{E}), \theta)$ is numerically Higgs flat and polystable and, by functoriality of the correspondence in \cite{PZZ_ProjectivelyFlatSimpsonCorrespondence}, the corresponding representation
        $\psi\colon \pi_1(X) \rightarrow \mathrm{GL}_{(\nu +1)^2}(\C)$ satisfies $\psi \cong \sigma \otimes \sigma^\vee$.

        In any case, by Selberg's Lemma \cite[Thm.\ II]{Cassels_SelbergsLemma}, up to replacing $X$ by some finite \'etale cover, we may assume that the images of $\sigma$ and $\rho$ are torsion-free.
        In this situation, Eyssidieux \cite[Thm.\ 2]{Eyssidieux_LinearShafarevichConjecture} showed that the \emph{Shafarevich morphism} $f\colon X \rightarrow Y$ of $\rho\oplus \sigma$ exists. Concretely, $f$ is a fibration which is determined by the following property: $f$ contracts a closed subvariety $Z\subseteq X$ to a point if and only if the image of the compositions
        \[
        \pi_1(Z) \rightarrow \pi_1(X) \xrightarrow{\rho} \mathrm{GL}_{n-\nu}(\C) \quad \mathrm{and} \quad
        \pi_1(Z) \rightarrow \pi_1(X) \xrightarrow{\sigma} \mathrm{PGL}_{\nu +1}(\C)
        \]
        are both finite. Note that the Zariski closures of the images of $\rho, \sigma$ are semisimple algebraic group. Indeed, for $\rho$ this follows from \cite[Lem.\ 2.1, Prop.\ 2.6]{PereiraTouzet_FoliationsVanishingChernClasses} while for $\sigma$ it it true because $\sigma$ is irreducible, see \cite[Prop.\ 3.1.15]{Mul_Thesis}. In particular, by \cite[Thm.\ 1]{CampanaClaudonEyssidieux_LinearShafarevich}, any resolution of singularities of $Y$ is of general type.

        During the rest of this proof, the aim is, morally speaking, to show that $f$ is the Iitaka fibration of $X$. To this end, let $F$ be a general fibre of $f$.

        \vspace{0.5\baselineskip}
            \emph{Claim: $F$ is a finite \'etale quotient of an Abelian variety and $\dim F \leq n-\nu$.}
        \vspace{0.5\baselineskip}

        \noindent
        Assuming the \emph{Claim} for a moment, the theorem immediately follows. Indeed, let $\phi\colon \widetilde{X}\rightarrow X$ and $\varphi\colon \widetilde{Y} \rightarrow Y$ be resolutions of singularities such that the induced rational map $\widetilde{f}\colon \widetilde{X} \dashrightarrow \widetilde{Y}$ is everywhere defined. Note that a general fibre $\widetilde{F}$ of $\widetilde{f}$ is birational to a general fibre of $f$.
        Then, by \cite[Cor.\ 1.2]{Kawamata_MinimalModelsKodairaDimension}, we find that
        \begin{align}
            \begin{split}
                \kappa(X) 
             \geq & \ \kappa\big(\widetilde{X}\big)
            \geq \kappa\big(\widetilde{F}\big) + \kappa\big(\widetilde{Y}\big) 
            = 0 + \dim \widetilde{Y} \\
            & = \dim Y
            = \dim X - \dim F
            \geq n - (n-\nu(X))
            = \nu(X).
            \end{split}
            \label{eq-Iitaka-conjecture}
        \end{align}
        Then, \cite[Thm.\ 6.1]{Kawamata_NefAndGoodDivisorsAreAbundant} shows that $K_X$ is semiample.
        \vspace{0.5\baselineskip}

        It remains to verify the \emph{Claim}. To this end, note that, by definition of $f$,
        \[
        \pi_1(F) \rightarrow \pi_1(X) \xrightarrow{\psi} \mathrm{GL}_{(\nu +1)^2}(\C)
        \]
        has finite image. By Malcev's theorem \cite[Thm.\ 4.2]{Wehrfritz_InfiniteLinearGroups}, $\psi(\pi_1(X))$ is residually finite. Consequently, there exists a subgroup $H\subseteq \pi_1(X)$ of finite index such that 
        \begin{equation}
            \psi(H)\cap \psi\big(\pi_1(F)\big) = \{1\}.
            \label{eq-Malcev}
        \end{equation}
        Let $\pi\colon X'\rightarrow X$ be the finite \'etale cover corresponding to $H$. Let $f'\colon X'\rightarrow Y'$ be the Stein factorisation of $f\circ \pi$ and let $F'$ be a general fibre of $f'$. Then, by \eqref{eq-Malcev}, the composition
        \begin{equation}
          \psi' \colon \pi_1(F') \rightarrow \pi_1(X') \hookrightarrow \pi_1(X) \xrightarrow{\psi} \mathrm{GL}_{(\nu +1)^2}(\C)  
          \label{eq-abundance-54}
        \end{equation}
        vanishes. Set $\mathcal{F}' := \mathcal{F}|_{F'}$, $\mathcal{E}' := \mathcal{E}|_{F'}$ and $\theta' :=\theta|_{F'}$. Recall that, by construction, $(\mathcal{E}\mathrm{nd}(\mathcal{E}), \theta)$ is the numerically Higgs flat Higgs vector bundle associated, via the Non-Abelian Hodge correspondence \ref{thm-Non-Abelian-Hodge-Correspondence}, to $\psi$. Consequently,
        $
        (\mathcal{E}\mathrm{nd}(\mathcal{E}'), \theta') = (\mathcal{E}\mathrm{nd}(\mathcal{E}), \theta)|_{F'}
        $
        is numerically Higgs, associated to $\psi' = \psi|_{\pi_1(F')}$, which is trivial, see \eqref{eq-abundance-54}. We deduce that 
        \begin{equation}
           \mathcal{E}\mathrm{nd}(\mathcal{E}') \cong \big(\mathcal{O}_{F'}\big)^{\oplus(\nu+1)^2} 
           \quad \mathrm{and} \quad
           \theta' = 0.
           \label{eq-abundance-55}
        \end{equation}
        On the other hand, one easily verifies from \ref{con-simpson-cotangent-sheaf} that
        \begin{align*}
            \left(\mathcal{E}\mathrm{nd}(\mathcal{F}')\oplus \mathcal{F}' \oplus \big(\mathcal{F}'\big)^\vee \oplus \mathcal{O}_{F'}\right) \otimes \mathcal{T}_{F'}
            \xrightarrow{\theta'}
        \mathcal{E}\mathrm{nd}(\mathcal{F}')
        \oplus \mathcal{F}' 
        \oplus \big(\mathcal{F}'\big)^\vee
        \oplus \mathcal{O}_{F'}
        \end{align*}
        is given explicitly by the formula
        \[
        (A, \eta, v, f) \otimes w \mapsto \Big(\mathrm{id}\otimes r(\eta)(w), 0, f \cdot (r(-))(w) + (r(A(-)))(w), (r(\eta))(w)\Big),
        \]
        where 
        $
        r\colon \mathcal{F}'\hookrightarrow \Omega_{X'}^{1}\big|_{F'} \rightarrow \Omega^{1}_{F'}
        $
        is the restriction map. Since $\theta'=0$ by \eqref{eq-abundance-55}, we conclude that $r(\mathcal{F}') = 0$, i.e.\ that
        \[
        \mathcal{F}' \subseteq \mathcal{N}^\vee_{F'/X'} = \ker\Big(\Omega_{X'}^{1}\big|_{F'} \rightarrow \Omega^{1}_{F'}\Big).
        \]
        We deduce that $\rk\mathcal{F} \leq \rk \mathcal{N}^\vee_{F'/X'} = n - \dim F'$. Since $\rk\mathcal{F} = \nu$ by \eqref{eq-ses-123}, a simple rearrangement shows that 
        $
        n - \nu \geq \dim F' = \dim F
        $,
        as required. This proves the dimension estimate in the \emph{Claim}. It remains to verify that the fibres of $f$ are \'etale quotients of Abelian varieties. To this end, we consider once more the short exact sequence
        \[
        0\rightarrow \mathcal{F}' \rightarrow \Omega_{X'}^1|_{F'} \rightarrow \mathcal{Q}' \rightarrow 0.
        \]
        Recall that $\mathcal{Q}'$ is numerically flat. Moreover, by \eqref{eq-abundance-55}, the vector bundle $\mathcal{F}'$ is a direct summand of a trivial vector bundle. It follows that both $\mathcal{F}'$ and its dual are nef \cite[Thm.\ 6.2.12]{lazarsfeld_PositivityII}. Then $\mathcal{F}'$ and, consequently,  $\Omega_{X'}^1|_{F'}$ are numerically flat \cite[Thm.\ 6.2.12]{lazarsfeld_PositivityII}. As $\Omega_{X'}^1|_{F'}$ is an extension of $\Omega^1_{F'}$ and the trivial vector bundle $\mathcal{N}^\vee_{F'/X'}$, we conclude that $\Omega^1_{F'}$ is numerically flat \cite[Thm.\ 6.2.12]{lazarsfeld_PositivityII}. In view of \cite{LuTaji_QuasiEtaleQuotientsAbelianVarieties}, this shows that $F'$ and, consequently, $F$ are finite \'etale quotients of an Abelian variety, thereby finishing the proof of the \emph{Claim} and, hence, of the theorem.
    \end{proof}

\section{Uniformisation: proof of Theorem \ref{INTRO-thm-statement-light} and Theorem \ref{INTRO-thm-statement-strong}}
    \label{sec-proof-of-thm-B}

    \noindent
    The goal of this section is to prove Theorem \ref{INTRO-thm-statement-strong}. Note that, in view of Theorem \ref{thm-uniformisation-v=c-semiampleness-part}, this will immediately imply Theorem \ref{INTRO-thm-statement-light} as well.

    \begin{notation}
    \label{not-uniformisation-codimension-two}
        For the rest of this section, let us fix the following notation: Let $X$ be a log terminal projective variety such that $K_X$ is semiample. We assume that there exists an integer $2\leq k \leq \nu := \nu(X)$ such that
        \[
        \Big(2(\nu + 1)c_2(X) - \nu c_1(X)^2\Big)\cdot K_X^{k-2} \equiv 0.
        \]
        Let $f\colon X \rightarrow Y := X_{\mathrm{can}}$ denote the Iitaka fibration of $X$. Write $K_X \sim_\Q f^*H$ for some ample $\Q$-divisor $H$ on $Y$, 
        fix a sufficiently large and divisible integer $\ell \gg 1$ and let 
        \[
        S= D_1\cap\ldots \cap D_{k-2} \subseteq Y
        \]
        be the complete intersection variety for general hyperplanes $D_1, \ldots, D_{k -2} \in |\ell H|$.
        Denote
        \[
        Z := f^{-1}(S) = f^{-1}(D_1) \cap \ldots \cap f^{-1}(D_{k -2}) \subseteq X.
        \]
        Then $Z$ is a log terminal projective variety of dimension $n-k +2$ and $S$ is a normal projective variety of dimension $\nu - k + 2$. As
        $
        K_Z \sim_{\Q} (1+(k -2)\ell)\cdot K_X|_Z
        $,
        we see that $K_Z$ is semiample and that $f\colon Z \rightarrow S$ is the Iitaka fibration of $Z$. 
    \end{notation}

    \begin{lem}
    \label{lem-uniformisation-structure-cotangent-bundle}
        In the notation of \ref{not-uniformisation-codimension-two}, let $\pi\colon Z'\rightarrow Z$ be a maximally quasi-\'etale cover. Then $Z'$ is smooth. Moreover, there exists a short exact sequence of vector bundles
        \[
        0 \rightarrow \mathcal{F}' \rightarrow \Omega^{[1]}_X\big|_{Z'} \rightarrow \mathcal{Q}' \rightarrow 0,
        \]
        such that
        \begin{itemize}
            \item[(1)] the Higgs vector bundle $(\mathcal{E}', \theta') := (\mathcal{F}' \oplus \mathcal{O}_{Z'}, \theta_{\Simp})$ is numerically projectively Higgs flat and stable; moreover, $\rk\mathcal{F}' =\nu$.
            \item[(2)] $\mathcal{Q}'$ is numerically flat.
        \end{itemize}
        In particular, $\Omega^1_{Z'}$ is nef.
    \end{lem}
    \begin{proof}
        According to Lemma \ref{lem-uniformisation-light}, there exists a short exact sequence
        \[
        0 \rightarrow \mathcal{F}_X \rightarrow \Omega_X^{[1]} \rightarrow \mathcal{Q}_X \rightarrow 0,
        \]
        where $\mathcal{F}_X$ is a reflexive sheaf of rank $\nu$, $\mathcal{Q}_X$ is torsion-free and the Chern classes of $\mathcal{F}_X$ and  $\mathcal{Q}_X$ satisfy various numerical constraints. Restricting to $Z$, we obtain the short exact sequence
        \begin{equation}
            0 \rightarrow \mathcal{F} \rightarrow \Omega_X^{[1]}\big|_Z \rightarrow \mathcal{Q} \rightarrow 0,
            \label{eq-uniformisation-codimension-two-ses}
        \end{equation}
        where $\mathcal{F} := \mathcal{F}_X|_Z$ and $\mathcal{Q} := \mathcal{Q}_X|_Z$. We can arrange for $\mathcal{F}$ to be reflexive and for $\mathcal{Q}$ to be torsion-free \cite[Cor.\ 1.1.14]{huybrechtsLehn_ModuliOfSheaves}. We will also require the Higgs subsheaf $(\mathcal{E}_X := \mathcal{F}_X \oplus \mathcal{O}_X, \theta_\Simp) \subseteq (\Omega_X^{[1]}\oplus \mathcal{O}_X, \theta_\Simp)$, see Construction \ref{con-simpson-cotangent-sheaf},
        and its restriction to $Z$:
        \begin{equation}
            \mathcal{E} := \mathcal{F} \oplus \mathcal{O}_Z, \quad
        \theta\colon \mathcal{E}\otimes\mathcal{T}_Z 
        \hookrightarrow \mathcal{E}\otimes\mathcal{T}_X\big|_Z
        \xrightarrow{\vartheta_\Simp} \mathcal{E},
        \label{eq-uniformisation-codimension-two-Higgs}
        \end{equation}
        which is itself a reflexive coherent Higgs sheaf on $Z$. We will divide the proof of Proposition \ref{prop-uniformisation-over-complete-intersection-surface} into a series of steps.
        
        \vspace{0.5\baselineskip}
            \emph{Step 1: $\mathcal{Q}^{\vee\vee}$ is numerically flat and $(\mathcal{E}, \theta)$ is numerically projectively Higgs flat and stable.}
            
         \vspace{0.5\baselineskip}   

        \noindent
        By Lemma \ref{lem-uniformisation-light} and the Mehta--Ramanathan restriction theorem for basepoint free linear series \cite[Thm.\ 7.3]{Langer_BogomolovsInequalityHiggsSheavesNormalVarieties}, $\mathcal{Q}$ is semistable with respect to any multipolarisation on $Z$. Moreover, by Lemma \ref{lem-uniformisation-light}, we have the following numerical identities:
        \begin{gather}
            c_2(\mathcal{F}) \equiv c_2(X),\quad
            c_1(\mathcal{F}) \equiv K_X,\quad
            c_2(\mathcal{Q}^{\vee\vee}) \equiv 0\quad \mathrm{and} \quad
            c_1(\mathcal{Q}) \equiv 0.
            \label{eq-uniformisation-codimension-two-numerics}
        \end{gather}
        In particular, $\mathcal{Q}^{\vee\vee}$ is numerically flat and the sequence
        \begin{equation}
            0 \rightarrow \mathcal{F} \rightarrow \Omega_X^{[1]}\big|_Z \rightarrow \mathcal{Q} \rightarrow 0
            \label{eq-uniformisation-codimension-two-ses-3}
        \end{equation}
        is Zariski locally split in codimension 2, see \cite[Lem.\ 2.13]{IMM_3c2=c1^2}. 
        Taking into account \eqref{eq-uniformisation-codimension-two-numerics}, it remains to verify that $(\mathcal{E}, \theta)$ is stable with respect to some multipolarisation on $Z$. Indeed, pick an ample divisor $H$ on $X$ and set $\alpha:= H|_Z^{n-k+1}$.
        By Lemma \ref{lem-uniformisation-light}, $(\mathcal{E}_X, \theta_{\Simp})$ is $(K_X^{k-2}\cdot H^{n-k+1})$-stable.
        Pick a proper Higgs subsheaf $\mathcal{K} \subseteq (\mathcal{E}, \theta)$. If the image of the projection
        \begin{equation}
        p\colon \mathcal{K} \hookrightarrow \mathcal{E} = \mathcal{F} \oplus \mathcal{O}_Z \rightarrow \mathcal{O}_Z
            \label{eq-projection}
        \end{equation}
        is trivial, then $\mathcal{K} \subseteq \mathcal{F}\subseteq \Omega_X^{[1]}|_Z$. In fact, since $\mathcal{K}$ is a Higgs subsheaf, we have that
        \[
        \mathcal{K} \subseteq \ker\left(\Omega_X^{[1]}\big|_Z \rightarrow \Omega_Z^{[1]}\right) = \mathcal{N}^\vee_{Z/X} \cong \bigoplus_{i=1}^{k-2} \mathcal{O}_Z(-\ell K_X|_Z).
        \]
        Since $K_X$ is nef, this implies that $\mu_{\alpha}(\mathcal{K}) \leq 0 < \mu_{\alpha}(\mathcal{E})$. On the other hand, if the image of \eqref{eq-projection} is non-trivial, then we compute
        \[
        \mu_{\alpha}(\mathcal{K}) 
        = \frac{c_1(\mathcal{K})\cdot \alpha}{\rk\mathcal{K}} 
        \leq \frac{c_1(\ker p)\cdot \alpha}{\rk\mathcal{K}} 
        = \mu_{\alpha}\big(\ker p \oplus \mathcal{O}_Z\big).
        \]
        Note that $\mathcal{K}' := \ker p \oplus \mathcal{O}_Z$ satisfies $\theta(\mathcal{K'}\otimes \mathcal{T}_X|_Z) \subseteq \mathcal{K}'$. Then, by \cite[Thm.\ 4.16]{Langer_BogomolovsInequalityHiggsSheavesNormalVarieties}, $\mu_{\alpha}(\mathcal{K}') < \mu_{\alpha}(\mathcal{E})$.
        In summary, $(\mathcal{E}, \theta)$ is $\alpha$-stable and, consequently, numerically projectively Higgs flat. This finishes the proof of \emph{Step 1}. 

        \vspace{0.5\baselineskip}
            \emph{Step 2: Let $\pi\colon Z'\rightarrow Z$ be a maximally quasi-\'etale cover. Then $Z'$ is smooth.}
        \vspace{0.5\baselineskip}

        \noindent
        Denoting
        $\mathcal{F}' := \pi^{[*]}\mathcal{F}$,
        $\mathcal{E}' := \pi^{[*]}\mathcal{E}$, 
        $\mathcal{Q}' := \pi^{[*]}\mathcal{Q}$, and
        $\Omega_X^{[1]}\big|_{Z'} := \pi^{[*]}(\Omega_X^{[1]}|_Z)$,
        we have the following sequence of sheaves on $Z'$:
        \begin{equation}
            0 \rightarrow \mathcal{F}' \rightarrow \Omega_X^{[1]}\big|_{Z'} \rightarrow \mathcal{Q}' \rightarrow 0.
            \label{eq-uniformisation-codimension-two-ses-2}
        \end{equation}
        It follows from \emph{Step 1} that $\mathcal{Q}'$ is numerically flat and that $(\mathcal{E}', \theta')$ is numerically projectively Higgs flat and polystable, see \cite[Lem.\ 3.16, Prop.\ 5.19, Prop.\ 5.20]{GKPT_MY_Inequality_Uniformisation_of_Canonical_models}.
        Moreover, \eqref{eq-uniformisation-codimension-two-ses-2} is exact in codimension 2, see \eqref{eq-uniformisation-codimension-two-ses-3}.
        Theorem \ref{thm-Non-Abelian-Hodge-Correspondence} now shows that the sheaves
        \[
        \mathcal{Q}' \quad \mathrm{and} \quad
        \mathcal{E}\mathrm{nd}(\mathcal{E}') 
        = \mathcal{E}\mathrm{nd}(\mathcal{F}' \oplus \mathcal{O}_{Z'})
        \cong \mathcal{E}\mathrm{nd}(\mathcal{F}') \oplus \mathcal{F}' \oplus \big(\mathcal{F}'\big)^\vee \oplus \mathcal{O}_{Z'}
        \]
        are locally free. In particular, also $\mathcal{F}'$ is locally free \cite[\href{https://stacks.math.columbia.edu/tag/00NX}{Lem.\ 00NX}]{stacks-project}. Then, according to \cite[Lem.\ 9.9]{AD_FanoFoliations}, the sequence \eqref{eq-uniformisation-codimension-two-ses-2} is exact and $\Omega_X^{[1]}|_{Z'}$ is locally free. 
        
        Recall that $Z\subseteq X$ is a complete intersection; hence, we have a locally split short exact sequence
        \[
        0 \rightarrow \mathcal{N}^\vee_{Z/X} \rightarrow \Omega_X^{1}|_Z \rightarrow \Omega_Z^{1} \rightarrow 0.
        \]
        This indicates that, locally, $\Omega_Z^{[1]}$ is a direct summand of $\Omega_X^{[1]}|_Z$. Consequently, locally, $\Omega_{Z'}^{[1]}$ is a direct summand of $\Omega_X^{[1]}|_{Z'}$. It follows that $\Omega_{Z'}^{[1]}$ is locally free \cite[\href{https://stacks.math.columbia.edu/tag/00NX}{Lem.\ 00NX}]{stacks-project}. Finally, by the resolution of the Lipman--Zariski conjecture for klt spaces \cite[Thm.\ 6.1]{GKKP_DifferentialFOrmsLogCanonicalSpaces}, we conclude that $Z'$ is smooth, thereby finishing the proof of \emph{Step 2}.

        \vspace{0.5\baselineskip}
            \emph{Step 3: The cotangent bundle $\Omega_{Z'}^1$ is nef.}
        \vspace{0.5\baselineskip}

        \noindent
        By \eqref{eq-uniformisation-codimension-two-ses-2}, we have the short exact sequence of vector bundles
        \[
        0 \rightarrow \mathcal{F}' \rightarrow \Omega_X^{[1]}\big|_{Z'} \rightarrow \mathcal{Q}' \rightarrow 0
        \]
        on $Z'$. By \emph{Step 1}, $(\mathcal{E}\mathrm{nd}(\mathcal{E}'), \theta')$ is numerically Higgs flat. Now, one easily verifies from \ref{con-simpson-cotangent-sheaf},
        that the sub vector bundle
        \[
        \big(\mathcal{F}'\big)^\vee 
        \subseteq \mathcal{E}\mathrm{nd}(\mathcal{F}') \oplus \mathcal{F}' \oplus \big(\mathcal{F}'\big)^\vee \oplus \mathcal{O}_{Z'}
        \cong \mathcal{E}\mathrm{nd}(\mathcal{E}') 
        \]
        is contained in the kernel of the Higgs field. Then, by Proposition \ref{prop-kernel-Higgs-field-seminegative}, $\mathcal{F}' = (\mathcal{F}')^{\vee\vee}$ is nef. Similarly, $\mathcal{Q}'$ is nef, since it is numerically flat. Consequently, also $\Omega_X^{[1]}|_{Z'}$ is nef \cite[Thm.\ 6.2.12]{lazarsfeld_PositivityII}. Finally, we observe that $\Omega_{Z'}^1$ is nef as it is a quotient of $\Omega_X^{[1]}|_{Z'}$, see \cite[Thm.\ 6.2.12]{lazarsfeld_PositivityII}.
    \end{proof}

    \begin{prop}
    \label{prop-uniformisation-over-complete-intersection-surface}
        In the notation above, there exists a commutative diagram
        \[\begin{tikzcd}
        Z' \arrow[swap]{r}{\varpi} \arrow[bend left=25]{rr}{\pi} \arrow[swap]{rd}{g} &  Z_{S'} \arrow[swap]{r}{\psi'} \arrow{d}{f'} & Z \arrow{d}{f} \\
        & S' \arrow{r}{\psi} & S
        \end{tikzcd}
        \]
        with the following properties:
        \begin{itemize}
            \item[\emph{(1)}] The horizontal arrows are all finite and Galois; moreover, $\pi, \varpi$ and $\psi'$ are quasi-\'etale.
            \item[\emph{(2)}] The varieties $Z'$ and $S'$ are smooth and $Z_{S'}$ coincides with the normalisation of the fibre product $Z\times_S S'$.
            \item[\emph{(3)}] $f'\colon Z_{S'} \rightarrow S'$ is a topological fibre bundle with connected fibres and $g\colon Z'\rightarrow S'$ is a holomorphic fibre bundle with Abelian varieties as fibres.
        \end{itemize}
    \end{prop}
    Recall, that a holomorphic map $f\colon W \rightarrow V$ between complex analytic varieties is a \emph{topological fibre bundle} if for every $x\in W$ there exists an open neighbourhood $f(x) \in U \subseteq V$ such that $f^{-1}(U)$ is homeomorphic to $U\times f^{-1}(f(x))$ over $U$. It is called a \emph{holomorphic fibre bundle} if $f^{-1}(U)\cong U\times f^{-1}(f(x))$ over $U$ as complex analytic varieties.
    \begin{proof}
        Let $\pi\colon Z'\rightarrow Z$ be a maximally quasi-\'etale cover and let
        \[\begin{tikzcd}
        Z' \arrow{r}{\pi} \arrow[swap]{d}{g} & Z \arrow{d}{f} \\
        S' \arrow{r}{\psi} & S
        \end{tikzcd}
        \]
        be the Stein factorisation of $f\circ \pi$. Since $f\colon Z \rightarrow S$ is the Iitaka fibration of $Z$, $g\colon Z' \rightarrow S'$ is the Iitaka fibration of $Z'$. By Lemma \ref{lem-uniformisation-structure-cotangent-bundle}, $\Omega_{Z'}^1$ is nef. Then, according to \cite[Thm.\ 1.2]{Hoering_ManifoldsNefCotangentBundle}, up to replacing $Z'$ by some finite \'etale cover, $g$ is a submersion with Abelian varieties as fibres. In particular, $S'$ is smooth. Moreover, we may assume that $g$ admits a section $\sigma \colon S'\rightarrow Z'$, see \cite[Lem.\ 2.1.(3)]{Hoering_ManifoldsNefCotangentBundle}.

        Next, let us show that $g\colon Z'\rightarrow S'$ is a holomorphic fibre bundle. Indeed, by Ambro's canonical bundle formula \cite[Proof of Thm.\ 2.16]{LazicFloris_CanonicalBundleFormula}, see also \cite[Thm.\ 0.2]{Ambro_CanonicalBundleFormula}, there exists an effective divisor $D_{S'}$ on $S'$ such that the pair $(S', D_{S'})$ is klt and such that
        \begin{equation}
            K_{Z'} \sim_\Q g^*(K_{S'} + D_{S'}).
            \label{eq-uniformisation-over-complete-intersection-surface-can-bun-form}
        \end{equation}
        In particular, the divisor $K_{S'} + D_{S'} = H'$ is ample. Let $\gamma\colon S'' \rightarrow (S', D_{S'})$ be an adapted cover \cite[Prop.\ 2.38]{ClaudonKebekusTaji_GenericSemipositivity}. It follows from \cite[Thm.\ C]{GuenanciaTaji_SemistabilityLogCotangentSheaf}, that the Higgs sheaf $(\Omega^1_{(S'', D_{S'}, \gamma)}\oplus\mathcal{O}_{S''}, \theta_{\Simp})$ is $((H')^{\nu-k+2})$-stable. 
        
        Now, recall that, according to Lemma \ref{lem-uniformisation-structure-cotangent-bundle}, there exists a short exact sequence of Higgs vector bundles on $Z'$:
        \[
        0 \rightarrow (\mathcal{E}', \theta') \rightarrow \Big(\Omega^{[1]}_X\oplus \mathcal{O}_X, \theta_{\Simp}\Big)\Big|_{Z'}
        \rightarrow (\mathcal{Q}', 0)
        \rightarrow 0
        \]
        Recall also, that $(\mathcal{E}', \theta')$ is Higgs numerically projectively flat and stable. Then also $\gamma^*\sigma^*(\mathcal{E}', \theta')$ is Higgs nummerically flat and polystable \cite[Thm.\ 1.5]{JahnkeRadloff_ProjFlatPairs}. Consider the composition
        \begin{align*}
            \begin{split}
                s\colon\gamma^*\sigma^*(\mathcal{E}', \theta')
                \hookrightarrow
                &\gamma^*\sigma^*\big(\Omega^{[1]}_X\oplus \mathcal{O}_X, \theta_{\Simp}\big)\Big|_{Z'} 
                \twoheadrightarrow
                 \gamma^*\sigma^*\big(\Omega_{Z'}^1 \oplus \mathcal{O}_{Z'}, \theta_{\Simp}\big)\\
                &\quad \twoheadrightarrow
                \gamma^*\big(\Omega^1_{S'}\oplus \mathcal{O}_{S'}, \theta_{\Simp}\big)
                \hookrightarrow
                \big(\Omega^1_{(S'', D_{S'}, \gamma)}\oplus \mathcal{O}_{S''}, \theta_{\Simp}\big).
            \end{split}
        \end{align*}
       As $\sigma^*(\mathcal{E}', \theta')$ is a polystable Higgs bundle and since $(\Omega^1_{S'}(\log D_{S'})\oplus \mathcal{O}_{S'}, \theta_{\Simp})$  is a stable Higgs bundle, see Proposition \ref{Prop-Properties-of-Simpsons-Higgs-Sheaf}, of the same slope, $s$ is either an isomorphism or the zero map \cite[Prop.\ 1.2.7]{huybrechtsLehn_ModuliOfSheaves}. Note that in the latter case, the induced morphism $\gamma^*\sigma^*(\mathcal{Q}', 0) \twoheadrightarrow (\Omega^1_{(S'', D_{S'}, \gamma)}\oplus \mathcal{O}_{S''}), \theta_{\Simp})$ would be generically surjective, which contradicts the fact that $\theta_{\Simp}$ is non-trivial. In effect, $s$ is an isomorphism. Since $s$ factors through $\gamma^*(\Omega^1_{S'}\oplus \mathcal{O}_{S'}, \theta_{\Simp})$, it follows that $D_{S'} = 0$. Consequently, $K_{Z'} \sim_\Q g^*K_{S'}$, see \eqref{eq-uniformisation-over-complete-intersection-surface-can-bun-form}.
        Then, it follows from \cite[Thm.\ 3.3]{Ambro_CanonicalBundleFormula}, that $g\colon Z'\rightarrow S'$ is generically a holomorphic fibre product. By \cite[Lem.\ 2.1.(4)]{Hoering_ManifoldsNefCotangentBundle}, up to replacing $Z'$ by some finite \'etale cover, $g$ is in fact a holomorphic fibre bundle, as required. Passing to some further cover, if necessary, we may also assume that $\pi\colon Z'\rightarrow Z$ is Galois, with Galois group $G$ say.

        Since $g$ is the Iitaka fibration of $Z'$, the action of $G$ on $Z'$ induces an action of $G$ on $S'$ such that $g$ becomes $G$-equivariant. Let $G_0\subseteq G$ be the kernel of the action of $G$ on $S'$. We obtain a factorisation
        \[\begin{tikzcd}
        Z' \arrow{r}{\varpi} \arrow[swap]{rd}{g} &  Z'/G_0 \arrow{r}{\psi'} \arrow{d}{f'} & Z \arrow{d}{f} \\
        & S' \arrow{r}{\psi} & S,
        \end{tikzcd}
        \]
        where $\varpi$ is finite Galois with Galois group $G_0$ and $\psi'$ is finite Galois with Galois group $G/G_0$ and $\pi = \psi'\circ \varpi$. Since $\pi$ is quasi-\'etale, so too are $\varpi$ and $\psi'$. Moreover, it follows from \cite[Thm.\ 2.4.B']{GP_RigidityOfLattices}, or, alternatively, the analytic slice theorem \cite[Prop.\ 2.5]{Snow_ReductiveGroupActionsOnSteinSpaces}, that $f'$ is a topological fibre bundle. Finally, $Z'/G_0$ may be identified with the normalisation of the fibre product $Z\times_S S'$; indeed, by universality, there is a natural finite morphism between both spaces and both are normal and finite Galois covers of $Z$ of the same degree.
    \end{proof}

    \noindent
    Next, we observe that the conclusion in Proposition \ref{prop-uniformisation-over-complete-intersection-surface} can be spread out over an open subset of the base of the Iitaka fibration of $X$:
    \begin{cor}
        \label{cor-uniformisation-codimension-two}
        In the notation of \ref{not-uniformisation-codimension-two}, there exists an open subset $U\subseteq Y$ with $\codim (Y\setminus U) \geq \nu - k + 3$ and a finite quasi-\'etale Galois cover $V' \rightarrow V := f^{-1}(U)$ with corresponding Stein factorisation
        \[\begin{tikzcd}
        V' \arrow{r}{\pi} \arrow[swap]{d}{g} & V \arrow{d}{f} \\
        U' \arrow{r}{\psi} & U
        \end{tikzcd}
        \]
        such that $V'$ and $U'$ are smooth. Moreover, we can arrange for there to exist an Abelian variety $A$ such that $V'\cong A\times U'$ over $U'$.
    \end{cor}
    \begin{proof}
        By the Lefschetz hyperplane theorem \cite[Sec.\ II.5.3]{GM_StratifiedMorseTheory}, the natural map
        \[
        \pi_1\big(S_{\reg}\setminus (S_{\reg} \cap D_f)\big)\rightarrow \pi_1\big(Y_{\reg}\setminus (Y_{\reg} \cap D_f)\big)
        \]
        is an isomorphism, where $D_f$ is the discriminant divisor of $f$. Hence, the finite Galois cover $\psi\colon S'\rightarrow S$ in Proposition \ref{prop-uniformisation-over-complete-intersection-surface} extends to a finite Galois cover $\psi\colon Y'\rightarrow Y$, which is \'etale over $Y_{\reg}\setminus D_f$, see also \cite[Prop.\ 3.13]{gkp_QuasiEtaleCovers}. Denote by $X_{Y'}$ the normalisation of the fibre product $X\times_Y Y'$ and let $f'\colon X_{Y'}\rightarrow Y'$ denote the induced morphism.
        
        Now, varying the general hypersurfaces $D_1, \ldots, D_{k -2} \in |\ell H|$ cutting out 
        $
        S= D_1\cap\ldots \cap D_{k-2} \subseteq Y,
        $
        we see that, by Proposition \ref{prop-uniformisation-over-complete-intersection-surface}, there exists an open subset $U\subseteq Y$ with $\codim (Y\setminus U) \geq \nu - k+3$ such that $U':= \psi^{-1}(U)$ is smooth and
        \[
        f'|_{U'}\colon V_{U'} := \big(f'\big)^{-1}\big(U'\big) \rightarrow U'
        \]
        is a topological fibre bundle. Hence, denoting by $F$ a general fibre of $f$, the sequence
        \[
        \pi_1\big(F_{\reg}\big) \rightarrow \pi_1\big((V_{U'})_{\reg}\big) \rightarrow \pi_1\big(U'\big) \rightarrow 0,
        \]
        is exact. Comparing with the analogous exact sequence
        \[
        \pi_1\big(F_{\reg}\big) \rightarrow \pi_1\big((Z_{S'})_{\reg}\big) \rightarrow \pi_1\big(S'\big) \rightarrow 0,
        \]
        and, noting that $\pi_1(U') \cong \pi_1(S')$, by Lefschetz, we see that the natural morphism
        \begin{equation}
            \pi_1\big((Z_{S'}\big)_{\reg}) \rightarrow \pi_1\big((V_{U'})_{\reg}\big)
            \label{eq-uniformisation-codimension-two-iso-pi1}
        \end{equation}
        is an isomorphism. Consequently, the finite quasi-\'etale cover $\varpi\colon Z' \rightarrow Z_{S'}$ in Proposition \ref{prop-uniformisation-over-complete-intersection-surface} extends to a finite quasi-\'etale cover $\varpi\colon V'\rightarrow V_{U'}$. Let us consider the composition 
        \[
        \pi\colon V'\xrightarrow{\varpi} V_{U'} \xrightarrow{\psi'} V:= f^{-1}(U).
        \]
        Varying again $S= D_1\cap\ldots \cap D_{k-2} \subseteq U$, we see that, up to shrinking $U$, the variety $V'$ is smooth and $g\colon V'\rightarrow U'$ is a holomorphic fibre bundle with an Abelian variety $A$ as fibres. But then, as in \cite[6.9]{Kollar_ShafarevichMapsAndPlurigenera}, one sees that, after possbly replacing $U'$ by some finite \'etale cover, $V'\cong A\times U'$, see \cite[Thm.\ 3.3.10]{Mul_Thesis} for the precise argument.
        Finally, by Proposition \ref{prop-Galois-covers-Abelian-varieties-ball-quotients}, we may assume that $\pi\colon V'\rightarrow V$ and, hence, $\psi\colon U'\rightarrow U$ are Galois. This concludes the proof.
    \end{proof}

    Now, by \cite[Thm.\ 3.8]{gkp_QuasiEtaleCovers}, there exists a unique normal projective compactification $U'\hookrightarrow Y'$ and a unique finite Galois morphism $\psi\colon Y'\rightarrow Y$ extending $\psi\colon U'\rightarrow U$. 
    To finish the proof of Theorem \ref{INTRO-thm-statement-strong}, essentially all that remains to do is to show that $Y'$ is a ball quotient; indeed, this will be an immediate consequence of the following Lemma:
    \begin{lem}
    \label{lem-uniformisation-codimension-two-base-of-Iitaka}
        In the notation above, the variety $Y'$ is log terminal and $K_{Y'}$ is ample.
    \end{lem}
    \begin{proof}
        Let $R_Y$ denote the branch divisor of $\psi\colon Y'\rightarrow Y$. Then the assertion of Lemma \ref{lem-uniformisation-codimension-two-base-of-Iitaka} is equivalent to verifying that the pair $(Y, R_Y)$ is klt and $K_Y + R_Y$ is ample, see \cite[Cor.\ 2.43]{Kollar_SIngularitiesOfTheMMP}. 

        Indeed, consider the divisor $B_Y := \sum_{D\subseteq Y} (1-\gamma_D) \cdot [D]$ on $Y$, where
        \[
        \gamma_D := \sup\left\{ t\in \R \big| (X, t\cdot f^*D) \textmd{ is log canonical near the generic point of } f^*D\ \right\}.
        \]
        By Ambro's canonical bundle formula \cite[Proof of Thm.\ 2.16]{LazicFloris_CanonicalBundleFormula}, see also \cite[Thm.\ 0.2]{Ambro_CanonicalBundleFormula}, there exists an effective $\Q$-divisor $M_Y$ on $Y$ such that the pair $(Y, B_Y + M_Y)$ is klt and satisfies
        \begin{equation}
            K_X \sim_\Q f^*(K_Y + B_Y + M_Y).
            \label{eq-uniformisation-codimension-two-base-of-Iitaka-2}
        \end{equation}
        Now, according to Corollary \ref{cor-uniformisation-codimension-two}, locally, over codimension one points of $Y$, the morphism $f\colon X \rightarrow Y$ is given by the quotient of the submersion $g\colon V'\rightarrow U'$ by a finite group $G$. Moreover, the action of $G$ on $V'$ is free in codimension one. Then it follows from  \cite[Cor.\ 2.43]{Kollar_SIngularitiesOfTheMMP} that $B_Y = R_Y$. Consequently, by \eqref{eq-uniformisation-codimension-two-base-of-Iitaka-2},
        \begin{equation}
            K_{V'} \sim_\Q \pi^*K_V \sim_\Q \pi^*f^*(K_Y + B_Y + M_Y) \sim_\Q g^*\psi^*(K_Y + R_Y + M_Y) \sim_\Q g^*(K_Y + M_Y).
            \label{eq-uniformisation-codimension-two-base-of-Iitaka-3}
        \end{equation}
        On the other hand,  by Corollary \ref{cor-uniformisation-codimension-two}, 
        \begin{equation}
            K_{V'} \sim_\Q g^*K_{U'}.
            \label{eq-uniformisation-codimension-two-base-of-Iitaka-4}
        \end{equation}
        Comparing \eqref{eq-uniformisation-codimension-two-base-of-Iitaka-3} and \eqref{eq-uniformisation-codimension-two-base-of-Iitaka-4}, and using that $\codim_Y (Y\setminus U) \geq 2$, we deduce that
        \[
        M_Y = 0.
        \]
        In summary, $(Y, R_Y)$ is klt and $K_Y + R_Y$ is ample, as required.
    \end{proof}
    
    We are now finally in position to complete the proof of Theorem \ref{INTRO-thm-statement-strong}:
    \begin{proof}[Proof of Theorem \ref{INTRO-thm-statement-strong}]
        Let $X$ be a log terminal projective variety of dimension $n$ such that $K_X$ is semiample. We assume that there exists an ample divisor $H$ on $X$ and an integer $2\leq k \leq \nu(X)$ such that
        \[
        \Big(2(\nu+1)c_2(X)  - \nu c_1(X)^2\Big)\cdot K_X^{i-2}\cdot H^{n-i} = 0, \quad \forall i=k, \ldots, n.
        \]
        Equivalently, by Theorem \ref{intro-thm-higher-inequalities-cycle-classes},
        \begin{equation}
        \Big(2(\nu+1)c_2(X)  - \nu c_1(X)^2\Big)\cdot K_X^{k-2} \equiv 0.
            \label{eq-uniformisation-codim-c}
        \end{equation}
        Denote by $f\colon X \rightarrow Y$ the Iitaka fibration of $X$. According to Corollary \ref{cor-uniformisation-codimension-two} there exists an open subset $U\subseteq Y$ with $\codim(Y\setminus U) \geq \nu - k + 3$, a finite, quasi-\'etale Galois cover $V'\rightarrow V := f^{-1}(U)$ and with corresponding Stein factorisation
        \[\begin{tikzcd}
        V' \arrow{r}{\pi} \arrow[swap]{d}{g} & V \arrow{d}{f} \\
        U' \arrow{r}{\psi} & U
        \end{tikzcd}
        \]
        and an Abelian variety $A$ such that
        \begin{equation}
            V' \cong A\times U'.\label{eq-uniformisation-codimension-two-2}
        \end{equation}
        By \cite[Thm.\ 3.8]{gkp_QuasiEtaleCovers}, there exists a unique normal projective compactification $U'\hookrightarrow Y'$ and a unique finite Galois morphism $\psi\colon Y'\rightarrow Y$ extending $\psi\colon U'\rightarrow U$.  According to Lemma \ref{lem-uniformisation-codimension-two-base-of-Iitaka}, $Y'$ is log terminal and $K_{Y'}$ is ample. Moreover, from \eqref{eq-uniformisation-codim-c}, \eqref{eq-uniformisation-codimension-two-2}, and the fact that $\codim(Y\setminus U)\geq 3$, we deduce that
        \[
        \Big(2(\nu+1)c_2\big(Y'\big) - \nu c_1\big(Y'\big)^2\Big) \cdot K_{Y'}^{\nu - 2} = 0.
        \]
        Then, by \cite[Thm.\ 1.5]{GKPT_HarmonicMetricsUniformisation}, there exists a finite quasi-\'etale cover $B\rightarrow Y'$ by a smooth ball quotient variety $B \cong \B^\nu/\Lambda$. Replacing $g\colon V'\rightarrow U'$ by its base change via $B\rightarrow Y'$, we may assume that $Y'\cong \B^n/\Lambda$ is a smooth ball quotient.

        Finally, denote by $G$ the Galois group of $\pi\colon V'\rightarrow V$ and set $U_B := (\psi')^{-1}(U')$. since $A\times B$ contains no rational curves, the action of $G$ on $A\times U_B$ extends to an action of $G$ on $A\times B$ by biregular automorphisms \cite[Cor.\ 1.5]{kollarMori_BirationalGeometry}. Then the birational map
        \[
        X\dashrightarrow (A\times B)/G
        \]
        induced by the isomorphism $V \cong (A\times U_B)/G$ satisfies all the assertions in Theorem \ref{INTRO-thm-statement-strong}.
    \end{proof} 

    \begin{proof}[Proof of Theorem \ref{INTRO-thm-statement-light}]
        According to Theorem \ref{thm-uniformisation-v=c-semiampleness-part}, the canonical divisor $K_X$ is semiample. Thus, Theorem \ref{INTRO-thm-statement-strong} applies and yields the conclusion.
    \end{proof}

    Let us end this section by proposing some questions:
    \begin{ques}
        Are the analogous statements of Theorem \ref{INTRO-thm-statement-light} and \ref{INTRO-thm-statement-strong} for klt pairs $(X, D)$ true? 
    \end{ques}
    Note that the generalisation of \eqref{INTRO-eq-Chern-class-inequality} to the setting of pairs should be easy to obtain, cf.\ Remark \ref{rem-Higher-Inequalities-klt-pairs}. However, treating the equality case seems more challenging. See \cite{CGG_UniformisationOrbifoldPairs,Dailly_FanosMYequality,GrafPatel_UniformisationByBoundedSymmetricDomains} for some related results in this direction.
    \begin{ques}
        In the situation of Theorem \ref{INTRO-thm-statement-strong}, is it true that the birational map $\varphi\colon X \dashrightarrow (A\times B)/G$ is everywhere defined, i.e.\ a morphism?
    \end{ques}
    Note that $(A\times B)/G$ may contain (many) rational curves, so \cite[Cor.\ 1.5]{kollarMori_BirationalGeometry} does not apply.
    \begin{ex}
        Let $B$ be a fake projective plane which admits an automorphism 
        $\sigma \colon B \rightarrow B$
        of order three. Then it is well-known that $B$ is a smooth projective ball quotient surface; moreover, $\sigma$ has precisely three fixed points \cite[Thm.\ 1.1]{Keum_QuotientsFakeProjectivePlanes}. 
       Let $E$ be an elliptic curve, set $A:= E\times E$ and let 
       \[
       \tau\colon A \rightarrow A, \quad (z_1, z_2)\mapsto (z_2, -(z_1 + z_2)),
       \]
       which is an automorphism of order three. Note that $\tau$ has precisely $9$ fixed points. The minimal resolution $\varsigma\colon S \rightarrow A/\langle \tau\rangle$ is a K3 surface of Picard rank $\rho(S) \geq 19$ \cite{Barth_K3Surfaces9cusps}. In particular, $S$ and, consequently, $A/\langle \tau\rangle$ contain infinitely many rational curves, see for example \cite{CGL_CurvesOnK3s}.

       Finally, set $X := (A\times B)/\langle \sigma\times \tau\rangle$. Then $X$ has $27$ singular points and contains infinitely many rational curves.
    \end{ex}

    \section{Proof of Theorem \ref{INTRO-thm-counterexample-catanese}}
\label{sec-examples}

    In this section, we prove Theorem \ref{INTRO-thm-counterexample-catanese} by explicitly constructing a variety $X$ with the required properties. Let us start by exhibiting a singular ball quotient $B = \B^4/\Lambda$ which will act as the canonical model $X_{\mathrm{can}} \cong B$. The following construction is a special case of a classical construction, employed previously by Kazhdan \cite{Khazdan_SomeApplicationsWeilRepresentations} and Rapoport--Zink \cite{RZ_LokaleZetaFunktionen}, see also \cite[Introduction]{Klingler_SymmetricDifferentialsKaehlerGroupsBallQuotients}:

    Set $K := \Q(\sqrt{3})$ and $L:= K(i)$. 
    On the $L$-vector space $V := L^{\oplus 5}$ we consider the anti-symmetric sesquilinear form $h\colon V\times V \rightarrow K$ determined by
        \[
        h\left( \left(
        \begin{matrix}
            x_1 + i y_1\\
            x_2 + i y_2\\
            x_3 + i y_3\\
            x_4 + i y_4\\
            x_5 + i y_5\\
        \end{matrix}
        \right)
        ,
        \left(
        \begin{matrix}
            x_1 + i y_1\\
            x_2 + i y_2\\
            x_3 + i y_3\\
            x_4 + i y_4\\
            x_5 + i y_5\\
        \end{matrix}
        \right)\right)
        := 
        \left(\sum_{i=1}^4 
        x_i^2+y_i^2\right)
        - 
        \sqrt{3} \cdot \Big(x_5^2+y_5^2\Big).
        \]
        Let $\mathcal{O}_L \subseteq L$ denote the ring of integers; concretely, $\mathcal{O}_L = \Z[\zeta_3, i]$, with $\zeta_3 := \frac{-1+\sqrt{3}}{2}$.
        We are going to construct the sought-after ball quotient $B$ as a quotient of $\B^4$ by a subgroup $\Lambda\subseteq \mathrm{SU}(h)$ commensurable with
        \[
        \Gamma := \mathrm{SU}(h) \cap \mathrm{SL}_5\left(\mathcal{O}_L\right) \subseteq \mathrm{SU}(h).
        \]
        Indeed, observe that $\Gamma \subseteq \mathrm{SU}(h)$ is discrete and finitely generated \cite{BorelHC_AithmeticSUbgroupsAlgebraicGroups}. Moreover, $\Gamma$ acts properly discontinuously on $\B^4$ through fractional linear transformations \cite[Lem.\ 9.3.(1)]{GKPT_MY_Inequality_Uniformisation_of_Canonical_models} and the quotient $\B^4/\Gamma$ is compact \cite[Rmk.\ 1.2]{Klingler_SymmetricDifferentialsKaehlerGroupsBallQuotients}. 
        However, the quotient $\B^4/\Gamma$ is not suitable for our purposes because the action of $\Gamma$ on $\B^4$ has fixed points in codimension one.
        
        By Selberg's lemma \cite[Thm.\ II]{Cassels_SelbergsLemma}, there exists a finite index normal subgroup $\Gamma_{\mathrm{tf}}\subseteq \Gamma$ which is torsion-free. Fix a prime ideal ${\mathfrak{p}} \subseteq \mathcal{O}_K$ such that $\zeta_3 \not\equiv 1$ modulo $\mathfrak{p}$ and set
        \[
        \Lambda_{\mathrm{tf}} := \big\{ \gamma \in \Gamma_{\mathrm{tf}} \ \big| \ \gamma \equiv \mathrm{Id} \quad \mathrm{mod}\ \mathfrak{p} \big\}
        \]
        Then $\Lambda_{\mathrm{tf}} \subseteq \Gamma_{\mathrm{tf}}$ is torsion-free and it is a normal subgroup of finite index in $\Gamma$. Hence, the action of $\Lambda_{\mathrm{tf}}$ on $\B^4$ is free, see for example \cite[Lem.\ 9.3.(3)]{GKPT_MY_Inequality_Uniformisation_of_Canonical_models}, and we deduce that
        \[
        B' := \B^4/\Lambda_{\mathrm{tf}}
        \]
        is a smooth projective ball quotient variety of dimension $4$. We wish to slightly enlarge $\Lambda_{\mathrm{tf}}$ to construct a singular quotient of $B'$. To this end, consider the element
        \[
        R := \left(
        \begin{matrix}
            \zeta_3 & 0 & 0 & 0 & 0 \\
            0 & \zeta_3 & 0 & 0 & 0 \\
            0 & 0 & \zeta_3 & 0 & 0 \\
            0 & 0 & 0 & 1 & 0 \\
            0 & 0 & 0 & 0 & 1 \\
        \end{matrix}
        \right)
        \in \Gamma.
        \]
        Note that $R^3 = 1$, so $R$ is torsion. We denote by $\Lambda\subseteq \Gamma$ the subgroup generated by $\Lambda_{\mathrm{tf}}$ and $R$; then it is easy to see that for any element $g\in \Lambda$ there exist unique elements $k\in \Lambda_{\mathrm{tf}}$ and $\delta\in \Z/3\Z$ such that $g = k\cdot R^\delta$, see \cite[Prop.\ 6.5.5]{Mul_Thesis} for details; in particular, $\Lambda_{\mathrm{tf}}\subseteq \Lambda$ is a subgroup of index three.

        We now turn our attention to the compact analytic variety
        \[
        B := \B^4/\Lambda = B'/G,
        \]
        where $G := \langle R\rangle \cong \Z/3\Z$. Our next order of business is to understand the fixed points of the action of $G$ on $B'$; in particular, we claim that there are none in codimension one. To this end, given $x\in \B^4$, we denote by $G_{[x]} \subseteq G$ the isotropy group of $[x] \in B'$. 
        We have a natural identification
        \[
        G_{[x]}\cong \Lambda_x := \{ \gamma \in \Lambda \ | \ \gamma \cdot x = x\} \subseteq \Lambda_{\mathrm{torsion}}.
        \]
        Now, let $x\in \B^4$ be a point such that $G_{[x]}$ is non-trivial. As $G\cong \Z/3\Z$ is simple, we deduce that $G_{[x]} \cong G$. Let $\gamma_x \in \Lambda_x$ be a generator for $G_{[x]}$. Then
        \begin{align}
            \gamma_x^3 = 1,
            \label{eq-generator-isotropy-group}
        \end{align}
        so $\gamma_x \notin \Lambda_{\mathrm{tf}}$. Consequently, there exists $\delta \in \{1, 2\}$ such that $\gamma_x \equiv R^\delta$ modulo $\mathfrak{p}$. Without loss of generality, we may assume that
        \begin{align}
            \gamma_x \equiv R \quad \mathrm{mod}\ \mathfrak{p}.
            \label{eq-reduction-of-stabiliser-mod-p}
        \end{align}
        \begin{prop}
        \label{prop-fixed-points}
            In the notation above, let $x \in Z_x \subseteq \B^4$ be the germ of $(\B^4)^{\Lambda_x}$ passing through $x$. Then $Z_x$ is either a smooth curve or a smooth surface.
        \end{prop}
        \begin{proof}
            By Luna's \'etale slice theorem \cite{luna_EtaleSlicing}, $Z_x$ is smooth and
            \begin{align}
                T_xZ_x = \ker\Big( d\gamma_x|_x - \mathrm{id}_{T_x\B^4}\Big).
                \label{eq-tangent-space-fixed-points}
            \end{align}
            In effect, we need to understand the eigenvalues of $d\gamma_x|_x$. Indeed, we consider the matrix $\gamma_x \in \Lambda \subseteq \mathrm{GL}_5(L)$. Recall that $\gamma_x^3 = 1$, see \eqref{eq-generator-isotropy-group}. Since $\zeta_3\in L$, we see that $\gamma_x$ is diagonalisable, i.e.\ there exists an $L$ basis $v_1, \ldots, v_5$ of eigenvectors for $\gamma_x$, of respective eigenvalue $\lambda_i \in \{1, \zeta_3, \zeta_3^2\}$. Replacing $v_i$ by $\mu \cdot v_i$ for some $\mu \in L$, we may assume that the entries of $v_i$ lie in $\mathcal{O}_L$ and that $v_i \not\equiv 0$ modulo $\mathfrak{p}$.
            
            Now, in view of \eqref{eq-reduction-of-stabiliser-mod-p},
            \[
            \lambda_i\cdot v_i = \gamma_xv_i \equiv Rv_i \quad \mathrm{mod}\ \mathfrak{p},
            \]
            so the reduction modulo $\mathfrak{p}$ of $v_i$ is an eigenvector for $R$. Consequently, the reductions modulo $\mathfrak{p}$ of the eigenvalues of $\gamma_x$ equal the eigenvalues of $R$. As $\zeta_3 \not\equiv 1$ modulo $\mathfrak{p}$ by assumption, we conclude that the eigenvalues of $\gamma_x$, with multiplicities, are given by the eigenvalues of $R$, namely $1, 1, \zeta_3, \zeta_3, \zeta_3$.

            Recall that $\Aut(\B^4) \cong \mathrm{SU}(h)$ acts transitively on $\B^4$. Hence, there exists $t\in \mathrm{SU}(h)$ such that $tx = 0\in \B^4$. Then the matrix $t\gamma_xt^{-1} \in \mathrm{SU}(h) \subseteq \mathrm{GL}_5(\C)$ is of the form
            \[
            t\gamma_xt^{-1} = 
            \left(
                \begin{matrix}
                V & 0  \\
                0 & w  \\
            \end{matrix}
            \right)
            \]
            for some unitary matrix $V \in U(4)$ and a complex number $w\in \C$. By the spectral theorem from linear algebra,  we can find $P \in \mathrm{SU}(4)$ such that the matrix $PVP^{-1}$ is diagonal. Set
            \[
            g := \left( \begin{matrix}
                P & 0 \\
                0 & 1
            \end{matrix}\right)\cdot t \in \mathrm{SU}(h).
            \]
            By construction, the matrix
            \[
            g\gamma_xg^{-1} = 
            \left(
                \begin{matrix}
                w_1 & 0 & 0 & 0 & 0 \\
                0 & w_2 & 0 & 0 & 0 \\
                0 & 0 & w_3 & 0 & 0 \\
                0 & 0 & 0 & w_4 & 0 \\
                0 & 0 & 0 & 0 & w_5 \\
                \end{matrix}
            \right)
            \]
            is diagonal. Then
            \begin{equation}
                g(d\gamma_x|_x)g^{-1} = d\big(g\gamma_xg^{-1}\big)\big|_0 =
            \left(
                \begin{matrix}
                \frac{w_1}{w_5} & 0 & 0 & 0 \\
                0 & \frac{w_2}{w_5} & 0 & 0 \\
                0 & 0 & \frac{w_3}{w_5} & 0 \\
                0 & 0 & 0 & \frac{w_4}{w_5} \\
                \end{matrix}
                \right).
                \label{eq-matrix}
            \end{equation}
            We conclude by observing that $w_1, \ldots, w_5$ are precisely the eigenvalues of $\gamma_x$. In particular, up to permutation, the numbers $w_1, \ldots, w_5$ are simply given by $1, 1, \zeta_3, \zeta_3, \zeta_3$. Consequently, by \eqref{eq-matrix}, the $1$-eigenspace of $d\gamma_x|_x$ has dimension one or two. The result then follows from \eqref{eq-tangent-space-fixed-points}.
        \end{proof}
        \begin{rem}
        \label{rem-fixed-points}
          Observe that $R$ acts on $\B^4$ via
            \[
            R \colon \B^4 \rightarrow \B^4, \quad (z_1, z_2, z_3, z_4) \mapsto \big(\zeta_3\cdot z_1, \zeta_3\cdot z_2, \zeta_3\cdot z_3, z_4\big).
            \]
            In particular, $\Lambda_0 = \langle R \rangle$ and we see that $0 \in Z_0 \subseteq \B^4$ is a smooth curve germ.
        \end{rem}

        \noindent
        We are finally in position to construct the variety $X$ in Theorem \ref{INTRO-thm-counterexample-catanese}:
        \begin{proof}[Proof of Theorem \ref{INTRO-thm-counterexample-catanese}]
        Set $Z := (B')^G$. Let $Z = Z_1\cup \ldots \cup Z_s$ denote the decomposition into connected components. By Proposition \ref{prop-fixed-points}, each component $Z_i$ is smooth and we may assume that there exists an integer $0\leq r\leq s$ such that $Z_1, \ldots, Z_r$ are smooth projective curves while $Z_{r+1}, \ldots, Z_{s}$ are smooth projective surfaces. In fact, $r\geq 1$ according to Remark \ref{rem-fixed-points}.

        Now, let $f'\colon X'\rightarrow B'$ be the blow-up along $Z_1, \ldots, Z_r$. 
        Then $X'$ is a smooth projective variety and, denoting by $E' = E_1' + \ldots + E_r'$ the exceptional divisor,
        \[
        K_{X'} \sim (f')^*K_{B'} + 2\cdot E'.
        \]
        Observe that the fixed points for the induced action of $G$ on $X'$ are given precisely by the (disjoint) union of the divisors $E'_i = \mP(\mathcal{N}_{Z_i/B'})$ and the strict transforms of the surfaces $Z_{r+1}, \ldots, Z_s$. The quotient $\pi\colon X'\rightarrow X := X'/G'$ is a log terminal projective variety and the ramification formula shows that
        \[
        K_{X'} \sim \pi^*K_X + 2\cdot E'.
        \]
        Consequently, denoting by $f\colon X \rightarrow B$ the induced birational morphism, it holds that $K_X \sim_\Q f^*K_B$.
        Since $K_B$ is ample, we deduce that $K_X$ is big and nef but not ample. Moreover, as $f$ is an isomorphism outside of the codimension three subset $Z_1 \cup \ldots \cup Z_r\subseteq B$, we conclude that
        \[
        \Big( 10c_2(X) - 4c_1(X)^2\Big)\cdot K^2_X = \Big( 10c_2(B) - 4c_1(B)^2\Big)\cdot K^2_B = 0.
        \]
        \end{proof}
        \begin{rem}
            While the fixed point locus $Z\subseteq B'$ always contains at least one curves, from the construction it seems plausible that $Z$ might not contain any surface, at least for suitable choice of $\Lambda_{\mathrm{tf}}$. If this were true, the quotient $X$ would even be smooth.
        \end{rem}
        \begin{ques}
            Can one find, for any $n\geq 3$, an example of a smooth projective varieties $X$ of dimension $n$ such that $K_X$ is big and nef but not ample and such that
            \[
            \Big(2(n+1)c_2(X) - nc_1(X)^2\Big)\cdot K_X^{n-2} = 0?
            \]
        \end{ques}
        Note that for $n\leq 2$ this is not possible, see \cite{Miyaoka_SurfacesWith3c2=c1^2}.


\bibliography{Literatur.bib}

 \end{document}